\documentclass[reqno,a4paper,11pt]{amsart}
\numberwithin{equation}{section}
\allowdisplaybreaks[1]

\usepackage[english]{babel}
\usepackage[utf8]{inputenc}
\usepackage[T1]{fontenc}
\usepackage{amsmath}
\usepackage{amssymb}
\usepackage{multirow}
\usepackage{alltt}
\usepackage{graphicx}
\usepackage{bbm}
\usepackage{amsthm}
\usepackage{mathrsfs}
\usepackage{textcomp}
\usepackage{calc}
\usepackage{fullpage}
\usepackage{amsrefs}
\usepackage{enumitem}

\usepackage[usenames,dvipsnames]{color}
\definecolor{darkblue}{RGB}{0,0,170}
\definecolor{brickred}{RGB}{200,0,0}
\usepackage[hyperindex, linktocpage]{hyperref}
\hypersetup{colorlinks=true, linkcolor=darkblue, citecolor=brickred}

\newcommand{\s}{\sigma}
\newcommand{\R}{\mathbb{R}}
\newcommand{\Z}{\mathbb{Z}}
\newcommand{\N}{\mathbb{N}}

\newcommand{\F}{\mathcal{F}}

\newcommand{\eps}{\varepsilon}
\newcommand{\dist}{\mathrm{dist}}
\newcommand{\supp}{\mathrm{supp}}
\newcommand{\diam}{\text{\normalfont diam}}
\newcommand{\pv}{\mathrm{p.\!v.}\!}
\newcommand{\Ds}{{\left(-\lapl\right)}^\s}

\newcommand{\lapl}{\triangle}
\newcommand{\grad}{\nabla}

\newcommand{\loc}{{\rm loc}}

\newcommand{\X}{{\mathcal{X}}}

\newcommand{\opnorm}[1]{{\left\vert\kern-0.25ex\left\vert\kern-0.25ex\left\vert #1 
    \right\vert\kern-0.25ex\right\vert\kern-0.25ex\right\vert}}
\DeclareMathOperator{\PV}{\pv}

\newtheorem{theorem}{Theorem}[section]
\newtheorem{lemma}[theorem]{Lemma}
\newtheorem{proposition}[theorem]{Proposition}
\newtheorem{corollary}[theorem]{Corollary}
\theoremstyle{definition}

\theoremstyle{remark}

\date{\today}
\author{Nicola Abatangelo}
\address{(N. Abatangelo) Institut für Mathematik, Goethe-Universität Frankfurt am Main, Robert-Mayer-Str. 10, 60325 Frankfurt am Main, Germany.}
\email{abatangelo@math.uni-frankfurt.de}
\author{Matteo Cozzi}
\address{(M. Cozzi) Dipartimento di Matematica ``Federigo Enriques'', Universit\`a degli Studi di Milano, Via Saldini 50, 20133 Milan, Italy \& Department of Mathematical Sciences, University of Bath, Claverton Down, Bath BA2 7AY, United Kingdom.}
\email{matteo.cozzi@unimi.it}

\title{An elliptic boundary value problem with fractional nonlinearity}



\thanks{{\it Acknowledgements}: The first author is supported by the Alexander von Humboldt Foundation. The second author has been supported by a Royal Society Newton International Fellowship.}

\begin{document}

\begin{abstract}
We investigate existence and uniqueness of solutions to second-order elliptic boundary value problems containing a power nonlinearity applied to a fractional Laplacian. 
We detect the critical power separating the existence from the non-existence regimes.
For the existence results, we make use of a particular class of 
weighted Sobolev spaces to compensate boundary singularities which are naturally built in the problem.
\end{abstract}

\maketitle

\section{Introduction}

Given a bounded domain~$\Omega\subseteq \R^N$ with~$\partial\Omega\in C^2$,~$\s\in(0,1)$, 
and~$p\in[1,\infty)$, we study problems of the form
\begin{equation}\label{prob}
\left\lbrace\begin{aligned}
-\lapl u+\big|\Ds u\big|^{p-1}\Ds u &= f & & \hbox{in }\Omega \\
u &= g & & \hbox{on }\partial\Omega \\
u &= h & & \hbox{in }\R^N\setminus\overline\Omega.
\end{aligned}\right.
\end{equation}
Here~$\lapl$ denotes the (classical) Laplace operator,
whereas~$\Ds$ is the fractional Laplacian
\begin{align}\label{def:Ds}
\Ds u(x):=c_{N,\s}\PV\int_{\R^N}\frac{u(x)-u(y)}{{|x-y|}^{N+2\s}}\,dy=
c_{N,\s}\lim_{\eps\downarrow 0}\int_{\R^N\setminus B_\eps(x)}\frac{u(x)-u(y)}{{|x-y|}^{N+2\s}}\,dy,
\end{align} 
a \textit{nonlocal} positive operator of \textit{fractional} order~$2\s\in(0,2)$.
The positive constant~$c_{N,\s}$ is a normalization in order to have
that the Fourier symbol of the operator is~$|\xi|^{2\s}$, \textit{i.e.},
\begin{align*}
\F\big[\Ds u\big](\xi)=|\xi|^{2\s}\F u(\xi),
\qquad u\in C^\infty_c(\R^N).
\end{align*}
We refer to~\cites{av,hitchhiker} for an introduction to this operator. 
Let us here simply remark that definition~\eqref{def:Ds} only makes pointwise sense 
for functions which are defined in the whole Euclidean space~$\R^N$.
For this reason, the prototypical well-posed boundary value problem driven by the fractional Laplacian
takes the form
\begin{align}\label{fractional Dirichlet}
\left\lbrace\begin{aligned}
\Ds u &= f & & \hbox{in }\Omega, \\
u &= h & & \hbox{in }\R^N\setminus\overline\Omega.
\end{aligned}\right.
\end{align}
In this setting, prescribing the values of the solution~$u$ on~$\partial\Omega$ is immaterial,
as the integral operator~$\Ds$ does not see negligible sets. Nevertheless,
it is reasonable to investigate the boundary regularity of solutions and, in particular,
their continuity across~$\partial\Omega$, like for example in~\cite{barles-chasseigne-imbert}
or more recently in~\cite{audrito-rosoton}---in both cases in a more general setting than the one in~\eqref{fractional Dirichlet}.

Problem~\eqref{prob} is motivated by the understanding of the interaction and the overlapping
of the different boundary conditions required by the Laplacian and the fractional Laplacian.
Indeed, although the term~$|\Ds u|^{p-1}\Ds u$ can be interpreted as a mere nonlinear perturbation
of the leading one~$\lapl u$, one needs to know the values of~$u$ also outside~$\Omega$
in order to make it meaningful. Mind also that different definitions of~$u$ outside~$\Omega$
might drastically affect the operator on the interior, a peculiarity of its nonlocality.

\subsection{Main results}

The exponent~$p$ has a prominent role in the solvability of~\eqref{prob}.
The range of~$p$ splits at~$p=1/\s$ into two regimes: below this value, which we call \textit{subcritical regime}, 
it is possible to solve~\eqref{prob}
if the other prescribed data~$f$ and~$g$ are somewhat well-behaved;
above that value, in the \textit{supercritical regime}, a switch in the lead of the equation takes place
and the nonlocal term becomes the dominant one, 
making it impossible (in general) to attain the desired values at the boundary~$\partial\Omega$.
See Theorems~\ref{thm:ex} and~\ref{thm:non-ex} for the precise statements.

In all the following,~$\Omega$ is supposed to be a bounded domain with $C^2$ boundary.

\begin{theorem}[Existence in the subcritical regime]\label{thm:ex}
Let~$g \in C^0(\partial\Omega)$,~$h \in L^\infty(\R^N\setminus\overline{\Omega})$, and~$f$ be a locally H\"older continuous function in~$\Omega$ satisfying
\begin{align}\label{hypo-data}
{\dist(\cdot,\partial\Omega)}^{2-\alpha}f\in L^\infty(\Omega),
\end{align}
for some~$\alpha > 0$. If
\begin{equation}\label{pqle1}
1\le p < \frac1\s,
\end{equation}
then problem \eqref{prob} has a unique solution~$u\in L^\infty(\R^N) \cap C^2(\Omega) \cap C^0(\overline\Omega)$.
\end{theorem}

We underline how data~$g$ and~$h$ are completely unrelated in the above statement:
in particular, we do not need to assume that the values of~$g$ on~$\partial\Omega$
match with those of a suitable extension of~$h$ to~$\partial\Omega$ 
(which, in fact, needs not to be continuous nor continuously extensible up
to the boundary).

In the particular case~$p=1$, the equation becomes linear and it admits a Green function
for which sharp two-sided estimates are available, see~Chen, Kim, Song, and Vondra\v cek~\cites{MR2928344, MR2912450}:
although we do not make use of these, let us just mention here that Theorem~\ref{thm:ex} (for~$p=1$) 
can also be deduced by means of such Green representation.
Very recently, Biagi, Dipierro, Valdinoci, and Vecchi~\cite{bdvv} have studied 
existence, maximum principles, and regularity for~\eqref{prob} with~$p=1$ and~$g,h=0$: 
their techniques and goals are quite different from our approach, 
but their main result \cite{bdvv}*{Theorem~1.7} is, nonetheless, closely related to Theorem~\ref{thm:ex}.

The core of the proof of Theorem~\ref{thm:ex} relies on a fixed-point argument
for the nonlinear operator
\begin{align}\label{nonlinear operator}
u \ \longmapsto\ -\lapl u + \big|\Ds u\big|^{p-1}\Ds u.
\end{align}
The (possible) jump discontinuity of~$u$, inherited by the prescription of~$g$ and~$h$ in \eqref{prob},
entails a singularity in the nonlocal part of~\eqref{nonlinear operator} at the boundary of~$\Omega$.
This represents a major challenge in solving~\eqref{prob}, which turns out to be not only a nonlinear problem but also a singular one.
To overcome this issue, we consider an approximating family of regularized problems, run the fixed-point argument to solve these problems, and pass to the limit via uniform estimates in Sobolev spaces of fractional order with boundary weights,
\textit{cf.}~\eqref{Lsptheta}-\eqref{normforLsptheta}.
Their definition is due to Lototsky~\cite{lot}:
we briefly outline their construction and prove some new results (\textit{cf}. Lemma~\ref{DsleH2slem}) in Section~\ref{sec:lototsky}.
The full proof of Theorem~\ref{thm:ex} is contained in Section~\ref{sec:ex}.

An important feature we will need in the proofs is a comparison principle.

\begin{proposition}[Weak comparison principle]\label{WCPprop}
Let~$\Omega \subseteq \R^N$ be a bounded open set and~$\Phi: \Omega \times \R \to \R$ be a Carath\'eodory function with~$\Phi(x, \cdot)$ non-decreasing for a.e.~$x \in \Omega$. Let~$\underline{w}, \overline{w} \in L^1_\s(\R^N) \cap C^2(\Omega)$ be two functions satisfying
\begin{equation} \label{WCPine}
\left\lbrace\begin{aligned}
- \lapl \underline{w} + \Phi\big(\, \cdot \,, \Ds \underline{w}\big) 
&\le - \lapl \overline{w} + \Phi\big(\, \cdot \,, \Ds \overline{w}\big) & & \mbox{in } \Omega \\
\underline{w} &\le \overline{w} & & \mbox{in } \R^N \setminus \overline{\Omega}.
\end{aligned}\right.
\end{equation}
For any~$x_0 \in \partial \Omega$, suppose in addition that
\begin{equation} \label{wover>wunderonboundary}
[-\infty, +\infty) \ni \limsup_{\Omega \ni x \rightarrow x_0} \underline{w}(x) \le \liminf_{\Omega \ni x \rightarrow x_0} \overline{w}(x) \in (-\infty, + \infty].
\end{equation}
Then,~$\underline{w} \le \overline{w}$ also in~$\Omega$.
\end{proposition}

Here,~$L^1_\s(\R^N)$ denotes the space of functions~$v\in L^1_{\loc}(\R^N)$ for which~$(1 + |\, \cdot \,|)^{- N - 2 \s} v \in L^1(\R^N)$. As is known, this assumption on~$v$, along with its local~$C^2$ (or~$C^{2 \s + \varepsilon}$) regularity, is enough to have~$\Ds v$ well-defined in the pointwise sense as the integral operator~\eqref{def:Ds}.
We will mostly apply Proposition~\ref{WCPprop} to functions~$\underline{w}$ and~$\overline{w}$ continuous up to the boundary of~$\Omega$. In this case, assumption~\eqref{wover>wunderonboundary} simply boils down to
\[
\underline{w} \le \overline{w} \quad \mbox{on } \partial \Omega.
\]
Notice, however, that~\eqref{wover>wunderonboundary} makes sense even when~$\underline{w}|_\Omega$ and~$\overline{w}|_\Omega$ cannot be continuously extended up to~$\partial \Omega$. This feature will be crucial in order to deal with solutions of~\eqref{prob} which blow up at the boundary.
Section~\ref{sec:CP} deals with the proof of Proposition~\ref{WCPprop} and its consequences.

Passing to non-existence results, we have the following.
\begin{theorem}[Non-existence in the critical and supercritical regimes]\label{thm:non-ex}
Let~$g \in C^0(\partial\Omega)$ and $h \in L^\infty(\R^N \setminus \overline{\Omega})$ 
be such that~$g \not\le 0$ on~$\partial \Omega$ and~$h \le 0$ in~$\R^N \setminus \overline{\Omega}$. 
If $\sigma p \geq 1$, then problem
\begin{equation}\label{probf=0}
\left\lbrace\begin{aligned}
-\lapl u+\big|\Ds u\big|^{p-1}\Ds u &= 0 & & \hbox{in }\Omega \\
u &= g & & \hbox{on }\partial\Omega \\
u &= h & & \hbox{in }\R^N\setminus\overline\Omega
\end{aligned}\right.
\end{equation}
has no solution~$u\in L^\infty(\R^N) \cap C^2(\Omega) \cap C^0(\overline\Omega)$.
\end{theorem}

Note that threshold~$p=1/\s$ is indeed quite natural. 
The analogue local problem
\begin{equation}\label{prob-clas}
\left\lbrace\begin{aligned}
-\lapl u+\big|\grad u\big|^{p} &= 0 & & \hbox{in }\Omega \\
u &= g & & \hbox{on }\partial\Omega
\end{aligned}\right.
\end{equation}
is known to always have a solution for $p\in[1,2]$~\cite{tomi}*{Hilfssatz 3} 
(or also~\cite{serrin}*{Section 11} for a more general statement),
whereas existence is lost in general for~$p>2$, see~\cite{serrin}*{Theorem~1, Section~16}.
Interestingly, in this case the critical power~$p=2$ is included in the
existence regime, whereas in~\eqref{prob}~$p=1/\s$
falls into the nonexistence one.

The next theorem shows that condition~\eqref{hypo-data} on the right-hand side is almost sharp.

\begin{theorem}[Non-existence for large sources]\label{thm:non-ex-sourc}
Let~$f\in L^\infty_{\loc}(\Omega)$ be such that
\[
{\dist(\cdot,\partial\Omega)}^{2}f \ge \kappa \quad\text{in }\Omega,
\]
for some~$\kappa > 0$. Then, for all~$p$ as in~\eqref{pqle1},
problem
\begin{equation}\label{probg,h=0}
\left\lbrace\begin{aligned}
-\lapl u+\big|\Ds u\big|^{p-1}\Ds u &= f & & \text{in }\Omega \\
u &= 0 & & \text{on }\partial\Omega \\
u &= 0 & & \text{in }\R^N\setminus\overline\Omega
\end{aligned}\right.
\end{equation}
has no solution~$u\in L^\infty(\R^N) \cap C^2(\Omega) \cap C^0(\overline\Omega)$.
\end{theorem}
Theorems~\ref{thm:non-ex} and~\ref{thm:non-ex-sourc} are proved in Section~\ref{sec:non-ex}.

We finally show, in Section~\ref{sec:large}, how the nonlinear character of~\eqref{prob}
allows for solutions that become singular\footnote{These solutions
appear in classical semilinear problems when the nonlinearity
fulfils the so-called \textit{Keller-Osserman condition}~\cites{keller,osserman}. 
They also show up in problems of the same type as~\eqref{prob-clas}
as remarked, for example, in~\cite{ll}. For fractional order equations the situation is more involved,
the interested reader might want to check~\cites{chen-felmer,fq,a1,a2}.} at~$\partial\Omega$
and are therefore called \textit{boundary blow-up solutions} or, simply, \textit{large solutions}.

\begin{theorem}[Large solutions]\label{theo:main} 
For any
\begin{equation*}
p\in\bigg(\frac{3-\s}{1+\s},\frac1\s\bigg),
\end{equation*}
problem 
\begin{align}\label{large-prob}
\left\lbrace\begin{aligned}
-\lapl u+\big|\Ds u\big|^{p-1}\Ds u &= 0 & & \hbox{in }\Omega \\
u &= \infty & & \hbox{on }\partial\Omega \\
u &= 0 & & \hbox{in }\R^N\setminus\overline\Omega
\end{aligned}\right.
\end{align}
admits a solution $u\in L^1(\R^N) \cap C^2(\Omega)$.
Moreover, there exists a~$C>0$ such that this solution satisfies
\begin{equation} \label{ugrowth}
0\ <\ u\ \leq\ C\,\dist{(\cdot,\partial\Omega)}^{-2(1-\s p)/(p-1)}\qquad \hbox{in }\Omega.
\end{equation}
\end{theorem}

\subsection{Notation}\label{notations}

We will use the following notation without further notice.

We set~$\delta=\dist(\cdot,\R^N\setminus\Omega)$ in~$\R^N$. By~$D^j$ we denote the collection of partial derivatives of order $j\in\N$.
The letters~$C$ and~$\overline{C}$ will be used to indicate constants that have values larger than~$1$ and that may change from line to line. Constants denoted by~$C$ will depend only on the structural quantities involved in the problem (\textit{i.e.},~$N$,~$\Omega$,~$p$,~$\s$, and~$\alpha$), whereas~$\overline{C}$ may also depend on the norms of the data---$\| \delta^{2 - \alpha} f \|_{L^\infty(\Omega)}$,~$\| g \|_{L^\infty(\partial \Omega)}$, and~$\| h \|_{L^\infty(\R^N \setminus \Omega)}$. 
Whenever a constant also depends on an additional quantity, we will emphasize it by using subscripts---for instance,~$C_j$ will denote a constant that also depends on~$j$ on top of the previously specified parameters.

\section{Fractional Sobolev spaces with weights}\label{sec:lototsky}

Following~\cite{lot}, we introduce weighted spaces~$L^{s,p}_\theta$ modelled upon Bessel potential spaces, from which we borrow the usual notation. 
For~$s \ge 0$ and~$p \ge 1$, we let~$L^{s, p}(\R^N)$ denote the Bessel potential space, 
obtained as the completion of~$C_c^\infty(\R^N)$ with respect to the norm
\begin{equation} \label{bespotnorm}
\| u \|_{L^{s, p}(\R^N)} := \big\|  {( 1 - \lapl  )}^{s/2} u \big\|_{L^p(\R^N)}.
\end{equation}
We recall that~$(1 - \lapl)^{s/2}$ is the operator defined by
\[
{(1-\lapl)}^{s/2} u := \F^{-1} \Big( \big(1 + |\cdot|^2 \big)^{s / 2} \F u \Big)
\quad \mbox{on any } u\in C^\infty_c(\R^N),
\]
where~$\F$ denotes the Fourier transform. 
Note that~$L^{0, p}(\R^N)$ reduces to the standard Lebesgue space~$L^p(\R^N)$.

For any~$k \in \Z$, we define
\[
A_k := \Big\{ x \in \Omega : 2^{- k - 1} < \delta(x) < 2^{- k + 1} \Big\}.
\]
Let~${(\zeta_k )}_{k \in \Z}$ be a smooth partition of unity, namely a family of non-negative functions~$\zeta_k \in C^\infty_c(\R^N)$ such that~$\supp(\zeta_k) \subseteq A_k$,
\[
|D^j \zeta_k(x)| \le C_{j} \, 2^{jk} \quad \mbox{for any } x \in \R^N \mbox{ and } j \in \N_{0},
\]
for some constant~$C_{|\alpha|} > 0$ and with~$|\alpha| = \alpha_1 + \ldots + \alpha_N$, and
\[
\sum_{k \in \Z} \zeta_k(x) = 1 \quad \mbox{for any } x \in \Omega.
\]

Given another parameter~$\theta \in \R$, introduce the space
\begin{equation}\label{Lsptheta}
L^{s, p}_\theta(\Omega) := \Big\{ u \in L^p(\Omega) : \| u \|_{L^{s, p}_\theta(\Omega)} < +\infty \Big\},
\end{equation}
where
\begin{equation} \label{normforLsptheta}
\| u \|_{L^{s, p}_\theta(\Omega)}^p := \sum_{k \in \Z} 2^{ - k \theta} \big\| \zeta_{k}(2^{- k} \cdot) u(2^{- k} \cdot) \big\|_{L^{s, p}(\R^N)}^p .
\end{equation}
For simplicity, we just write~$L^p_\theta(\Omega)$ instead of~$L^{0, p}_\theta(\Omega)$.

The space~$L^{s,p}_\theta(\Omega)$ is a Sobolev space with a weight at the boundary~$\partial\Omega$, introduced by Lototsky in~\cites{lot0,lot}.
In the next statement we collect some of its basic properties that will be used in the sequel---see~\cite{lot}*{Proposition~2.2} for their proofs.

\begin{proposition}[\cite{lot}*{Proposition 2.2}] \label{lotprops}
The following statements hold true.
\begin{enumerate}[label=$\roman*)$,leftmargin=3em]
\item The space~$C^\infty_c(\Omega)$ is dense in~$L^{s, p}_\theta(\Omega)$.
\item The space~$L^{s, p}_\theta(\Omega)$ is independent of the choice of partition of unity~${( \zeta_k)}_{k\in\Z}$, and different partitions lead to equivalent norms.
\item The quantity
\begin{align}\label{equivalent lototsky norm}
\| u \|_{L^{p}_\theta(\Omega)}^* := \Big( \int_\Omega \big|u(x)\big|^p \, {\delta(x)}^{\theta - N} \, dx \Big)^{1 / p}
\end{align}
defines an equivalent norm for the space~$L^p_\theta(\Omega)$.
\item If~$s = k$ is a non-negative integer, then
\[
L^{k, p}_\theta(\Omega) = \Big\{ u \in L^p_\theta(\Omega) : \delta^j D^j u \in L^p_\theta(\Omega) \mbox{ for all } j = 1, \ldots, k \Big\}
\]
and the norm
\[
\| u \|_{L^{k, p}_\theta(\Omega)}^* := \bigg( \sum_{j = 0}^k \big( \| \delta^j D^j u \|_{L^p_\theta(\Omega)}^* \big)^p \bigg)^{1/p}
\]
is equivalent to the one defined in~\eqref{normforLsptheta}.
\item For~$s_i \ge 0$,~$p_i > 1$,~$\theta_i \in \R$ ($i = 0, 1$), and~$\nu \in (0, 1)$, it holds
\[
L^{s, p}_\theta(\Omega) = \Big[ L^{s_0, p_0}_{\theta_0}(\Omega), L^{s_1, p_1}_{\theta_1}(\Omega) \Big]_\nu,
\]
with
\[
s := (1 - \nu) s_0 + \nu s_1, \quad \frac{1}{p} := \frac{1 - \nu}{p_0} + \frac{\nu}{p_1}, \quad \theta := (1 - \nu) \theta_0 + \nu \theta_1,
\]
and where~$[X, Y]_\nu$ denotes the complex interpolation space of~$X$ and~$Y$.
\end{enumerate}
\end{proposition}
In view of~\eqref{equivalent lototsky norm}, one has~$L^p_N(\Omega)=L^p(\Omega)$---note however that~$L^{s,p}_N(\Omega)$ differs from the unweighted Bessel potential space~$L^{s,p}(\Omega)$ when~$s>0$.

One of the most important outcomes of the analysis carried out in~\cite{lot} is a weighted Calder\'on-Zygmund-type estimate for solutions of a class of degenerate elliptic second-order equations. In Section~\ref{sec:ex}, we will take advantage of a very particular case of this result, which we state here below for the reader's convenience and for further reference---see~\cite{lot}*{Section~5} for its proof.

\begin{proposition}[\cite{lot}*{Lemma 5.2}] \label{lotest}
Let~$p > 1$,~$f \in L^p_\theta(\Omega)$, and~$u \in C^2(\Omega) \cap L^p_\theta(\Omega)$ be a solution of~$- \lapl u = f$ in~$\Omega$. Then,~$u \in L^{2, p}_\theta(\Omega)$ and
\[
\| u \|_{L^{2, p}_\theta(\Omega)} \le C \Big( \| f \|_{L^p_{\theta + 2 p}(\Omega)} + \| u \|_{L^p_\theta(\Omega)} \Big),
\]
for some constant~$C > 0$ depending only on~$N$,~$p$,~$\theta$, and~$\Omega$.
\end{proposition}

One last element of this theory that we will need in the sequel is an estimate on the~$L^{p}_\theta(\Omega)$ norm of the fractional Laplacian of a regular function. Before heading to its statement, we note that, for~$p > 1$, the norm~$\| \cdot \|_{L^{s, p}(\R^N)}$ introduced in~\eqref{bespotnorm} for~$L^{s, p}(\R^N)$ is equivalent to
\[
\opnorm{u}_{L^{s, p}(\R^N)} := \| u \|_{L^p(\R^N)} + \big\| (- \lapl)^{s/2} u \big\|_{L^p(\R^N)}. 
\]
Indeed, for~$s \in \N$ this follows from~\cite{mazya}*{Corollary~10.1.2/1} and~\cite{adams}*{Corollary~4.16 and Theorem~7.63(f)}. In this case,~$L^{s, p}(\R^N)$ coincides with the usual Sobolev space~$W^{s, p}(\R^N)$. When~$s$ is not an integer, the equivalence can be deduced from~\cite{mazya}*{Theorem~10.1.2/4}. 

Thanks to this observation, we can prove the following estimate. We will use it in Sections~\ref{sec:ex} and~\ref{sec:large}, respectively with~$r = p$ and~$r = 1$.

\begin{lemma} \label{DsleH2slem}
Let~$p > 1$,~$r \in [1, p]$,~$\s \in (0, 1)$, and~$\theta > (N / r  + 2 \s) p$. Then, for every~$\varepsilon > 0$, there exists a constant~$C > 0$, depending only on~$N$,~$p$,~$r$,~$\s$,~$\theta$,~$\Omega$, and~$\varepsilon$, such that
\begin{equation} \label{DsleH2s}
\| \Ds u \|_{L^p_\theta(\Omega)} \le C \Big( \| u \|_{L^{2 \s, p}_{\theta - 2 \s p - \varepsilon}(\Omega)} + \| u \|_{L^r(\Omega)} \Big)
\end{equation}
for every function~$u: \R^N \to \R$ that vanishes outside~$\Omega$ and is of class~$C^2$ in~$\Omega$.
\end{lemma}
\begin{proof}
By the properties of the~$\zeta_k$'s, the fact that~$u = 0$ in~$\R^N \setminus \Omega$, and a suitable change of variables, we have
\begin{equation} \label{nicest}
\begin{aligned}
\| \Ds u \|_{L^p_\theta(\Omega)} & \le \sum_{j \in \Z} \| \Ds ( \zeta_j u ) \|_{L^p_\theta(\Omega)} \\
& = \sum_{j \in \Z} \bigg( \sum_{k \in \Z} 2^{ - k \theta} \Big\| \zeta_{k}(2^{- k} \cdot) \Ds ( \zeta_j u )(2^{- k} \cdot) \Big\|_{L^p(\R^N)}^p \bigg)^{1/p}
\\
& \le \sum_{j \in \Z} \bigg( \sum_{k \in \Z} 2^{ - (\theta - N) k} \big\| \Ds ( \zeta_j u ) \big\|_{L^p(A_k)}^p \bigg)^{1/p}.
\end{aligned}
\end{equation}

Now, we consider separately the cases~$k \ge j - 2$ and~$k \le j - 3$. 

In the first situation, we estimate
\begin{equation} \label{kgejest}
\begin{aligned}
\sum_{k = j - 2}^{+\infty} 2^{ - (\theta - N) k} \big\| \Ds ( \zeta_j u ) \big\|_{L^p(A_k)}^p & \le 4^{\theta - N} 2^{ - (\theta - N) j} \sum_{k = j - 2}^{+\infty} \big\| \Ds ( \zeta_j u ) \big\|_{L^p(A_k)}^p \\
& \le 3 \cdot 4^{\theta - N} 2^{ - (\theta - N) j} \big\| \Ds ( \zeta_j u ) \big\|_{L^p(\R^N)}^p,
\end{aligned}
\end{equation}
where for the last inequality we used that the~$A_k$'s intersect at most three times. 

On the other hand, take~$k \le j - 3$ and observe that~$|x - y|\ge 2^{- k - 2}$ for every~$x \in A_k$ and~$y \in A_j$. We compute
\begin{multline*}
\big\| \Ds \big( \zeta_j u \big) \big\|_{L^p(A_k)} = \Bigg( \int_{A_k} \bigg| \int_{\R^N} \frac{\zeta_j(y) u(y)}{|x - y|^{N + 2 \s}} \, dy \, \bigg|^p dx \Bigg)^{1/p} \leq \\
\le  \Bigg( \int_{\R^N} \bigg( \int_{\R^{N}} \zeta_j(y) |u(y)| \frac{\chi_{\R^N \setminus B_{2^{- k - 2}}}(x - y)}{|x - y|^{N + 2 \s}} \, dy \, \bigg)^p dx \Bigg)^{1/p}.
\end{multline*}
From this and Young's convolution inequality, we deduce that
\begin{multline*}
\big\| \Ds ( \zeta_j u ) \big\|_{L^p(A_k)} \le 
\| \zeta_j u \|_{L^r(\R^N)} \bigg( \int_{\R^N \setminus B_{2^{- k - 2}}} |z|^{- \frac{(N + 2 \s) p r}{p r - p + r}} \, dz \bigg)^{\frac{p r - p + r}{p r}} \leq \\
\le 2^{\left( \frac{N}{r} - \frac{N}{p} + 2 \s \right) k} C\| \zeta_j u \|_{L^r(\R^N)}.
\end{multline*}
Also, since~$\Omega$ is bounded, there exists~$k_0\in\Z$ such that~$A_k = \varnothing$ for all~$k < k_0$. In light of these facts, we have
\[
\sum_{k = -\infty}^{j - 3} 2^{ - (\theta - N) k} \big\| \Ds ( \zeta_j u ) \big\|_{L^p(A_k)}^p  \le C \| \zeta_j u \|_{L^r(\R^N)}^p \sum_{k = k_0}^{j - 3}  2^{ - k( \theta - \frac{N p}{r} - 2 \s p )} \le C \| u \|_{L^r(A_j)}^p,
\]
where the last inequality holds since~$\theta > (N/r + 2 \s) p$.

By combining the above estimate with~\eqref{kgejest}, changing coordinates, and using the scaling property
\[
\Ds \big[ v(\lambda \cdot) \big] = \lambda^{2 \s} \Ds v(\lambda \cdot) \quad \mbox{for } \lambda > 0, 
\]
we deduce from~\eqref{nicest} that
\begin{multline*}
\big\| \Ds u \big\|_{L^p_\theta(\Omega)} \le C \sum_{j = k_0}^{+\infty} \Big( 2^{ - \frac{\theta - N}{p} j} \big\| \Ds \left( \zeta_j u \right) \big\|_{L^p(\R^N)} + \| u \|_{L^r(A_j)} \Big) \leq \\
\le C \bigg( \sum_{j = k_0}^{+\infty} 2^{ - \frac{\theta - 2 \s p}{p} j} \Big\| \Ds \big[ \zeta_j(2^{- j} \cdot) u(2^{- j} \cdot) \big] \Big\|_{L^p(\R^N)} + \| u \|_{L^r(\Omega)} \bigg).
\end{multline*}
Since, by the definition of~$\opnorm{\, \cdot \,}_{L^{2 \s, p}(\R^N)}$ and its equivalence to~$\| \cdot \|_{L^{2 \s, p}(\R^N)}$,
\begin{align*} 
\Big\| \Ds \big[ \zeta_j(2^{- j} \cdot) u(2^{- j} \cdot) \big] \Big\|_{L^p(\R^N)}\leq 
\opnorm{\zeta_j(2^{- j} \cdot) u(2^{- j} \cdot)}_{L^{2\s,p}(\R^N)}\leq 
C\big\|\zeta_j(2^{- j} \cdot) u(2^{- j} \cdot)\big\|_{L^{2\s,p}(\R^N)},
\end{align*}
we conclude that
\begin{align*}
\big\| \Ds u \big\|_{L^p_\theta(\Omega)} \le C \bigg( \sum_{j = k_0}^{+\infty} 2^{ - \frac{\theta - 2 \s p}{p} j} \big\| \zeta_j(2^{- j} \cdot) u(2^{- j} \cdot) \big\|_{L^{2 \s, p}(\R^N)} + \| u \|_{L^r(\Omega)} \bigg).
\end{align*}

To obtain~\eqref{DsleH2s}, we are left to show that the above series can be controlled by the~$L^{2 \s, p}_{\theta - 2 \s p - \varepsilon}(\Omega)$ norm of~$u$, for every~$\varepsilon > 0$. This is a consequence of a simple inequality for numerical series. Indeed, given any sequence~${(a_j)}_{j\in\N}$ of non-negative numbers, H\"older's inequality gives that
\begin{multline*}
\sum_{j = k_0}^{+\infty} 2^{ - \frac{\theta - 2 \s p}{p} j} a_j  
= \sum_{j = k_0}^{+\infty} \Big( 2^{ - \frac{\theta - 2 \s p - \varepsilon}{p} j} a_j \Big) 2^{- \frac{\varepsilon}{p} j} \\
\le \bigg( \sum_{j = k_0}^{+\infty} \Big( 2^{ - \frac{\theta - 2 \s p - \varepsilon}{p} j} a_j \Big)^p \bigg)^{1/p} \bigg( \sum_{j = k_0}^{+\infty} 2^{- \frac{\varepsilon}{p - 1} j} \bigg)^{(p - 1)/ p} 
\le C \bigg( \sum_{j \in \Z} 2^{ - (\theta - 2 \s p - \varepsilon) j} a_j^p \bigg)^{1/p}.
\end{multline*}
Applying this with~$a_j = \| \zeta_j(2^{- j} \cdot) u(2^{- j} \cdot) \|_{L^{2 \s, p}(\R^N)}$ and recalling~\eqref{normforLsptheta}, we infer that~\eqref{DsleH2s} holds true.
\end{proof}

\section{Barriers}

In this section, we present the construction of positive supersolutions for both the classical and the fractional Laplace operator in bounded domains. Ultimately, they will be used as barriers for the nonlinear operator
\[
u \ \longmapsto\ - \lapl u + \big|\Ds u\big|^{p - 1} \Ds u.
\]

We begin with a simple barrier that will be used for equations involving bounded right-hand sides. In this case, we may restrict ourselves to consider balls as underlying domains, and thus the construction is rather standard. As usual, for~$x, y > 0$ we denote by~$B(x, y) = \int_0^1 t^{x - 1} (1 - t)^{y - 1} \, dt$ the beta function.

\begin{lemma} \label{easybarlem}
For~$R > 0$, the function~$u_\s(x) := (R^2 - |x|^2)_+^\s,\ x\in\R^N$, satisfies
\begin{equation} \label{laplw2s}
- \lapl u_\s(x) = 2 \s \frac{N (R^2 - |x|^2) + 2 (1 - \s) |x|^2}{(R^2 - |x|^2)^{2 - \s}} \ge 2 \s N R^{2 \s - 2},
\qquad\text{for every }x\in B_R,
\end{equation}
and
\begin{equation} \label{laplfracw2s}
\Ds u_\s(x) = \frac{c_{N, \s} B(\s, 1 - \s) \, |\partial B_1|}{2} > 0,
\qquad\text{for every }x\in B_R.
\end{equation}
\end{lemma}
\begin{proof}
A straightforward computation gives~\eqref{laplw2s}. 
Identity~\eqref{laplfracw2s} is due to~\cite{getoor}---see~\cite{dyda} for an elementary proof and for more general relations.
\end{proof}

In order to deal with right-hand sides that blow up at the boundary, we can no longer limit ourselves to balls, and instead we need to construct barriers tailored to each specific domain. We do this in the next lemma, by considering powers of the so-called torsion function, \textit{i.e.}, of the solution of the Dirichlet problem
\begin{equation} \label{torsprob}
\left \lbrace
\begin{aligned}
- \lapl \tau & = 1 & & \hbox{in } \Omega,\\
\tau & = 0 & & \hbox{on } \partial \Omega.
\end{aligned}
\right.
\end{equation}
The existence and uniqueness of the solution~$\tau$ of the torsion problem~\eqref{torsprob} is classical. Furthermore,~$\tau > 0$ in~$\Omega$ thanks to the strong maximum principle. 

\begin{lemma} \label{lem:valpha}
Let~$\Omega \subseteq \R^N$ be a bounded domain with boundary of class~$C^2$ and~$\tau \in C^\infty(\Omega) \cap C^1(\overline{\Omega})$ be the solution of~\eqref{torsprob}. For~$\alpha \ge 0$, set
\[
v_\alpha := \tau^\alpha \chi_{\overline{\Omega}} \quad \mbox{in } \R^N.
\]
Then, the following statements hold true.
\begin{enumerate}[label=$\roman*)$,leftmargin=3em]
\item There exists a constant~$C_1 \ge 1$, depending only on~$\Omega$, such that
\begin{equation} \label{valpha-bounds}
C_1^{-1} \delta^\alpha \le v_\alpha \le C_1 \delta^{\alpha} \quad \mbox{in } \Omega,
\end{equation}
for all~$\alpha\in(0, 1]$.
\item There exists a constant~$C_2 \ge 1$, depending only on~$N$,~$\Omega$, and~$\s$, such that
\begin{align}
\label{laplvalpha}
C_2^{-1} \alpha (1 - \alpha) \, \delta^{\alpha - 2} \le - \lapl v_\alpha & \le C_2 \, \alpha \, \delta^{\alpha - 2} \quad \mbox{in } \Omega, \\
\label{fraclapvalpha}
- C_2 \, \diam(\Omega)^\alpha \delta^{- 2 \s} \le \Ds v_\alpha & \le C_2 \, \delta^{\alpha-2\s} \quad \hspace{3pt} \mbox{in } \Omega,
\end{align}
for all $\alpha\in (0, 1)$.
\item There exists a constant~$C_3 \ge 1$, depending only on~$N$ and~$\Omega$, such that
\begin{equation} \label{fraclapv0}
C_3^{-1} \delta^{-2\s} \le \Ds v_0 \le C_3 \, \delta^{-2\s} \quad \mbox{in } \Omega.
\end{equation}
\end{enumerate}
\end{lemma}
\begin{proof}
First, we notice that
\begin{equation} \label{taubounds}
C_1^{-1} \delta \le \tau \le C_1 \delta \quad \mbox{in } \Omega,
\end{equation}
for some constant~$C_1 \ge 1$ depending only on~$\Omega$. The upper bound on~$\tau$ follows from its~$C^1(\overline{\Omega})$ regularity, whereas the lower bound is a consequence of its positivity inside~$\Omega$ and of Hopf's lemma.

Knowing this, we address the validity of the claims made in the statement. Of course,~\eqref{valpha-bounds} follows from~\eqref{taubounds} right away. In order to prove~\eqref{laplvalpha}, a straightforward computation gives that
\begin{equation} \label{vijcomp}
\partial_{i j} v_\alpha = \alpha (\alpha - 1) \tau^{\alpha - 2} \partial_i \tau \partial_j \tau + \alpha \tau^{\alpha - 1} \partial_{i j} \tau \quad \mbox{in } \Omega, 
\end{equation}
for every~$i, j = 1, \ldots, N$, and thus, by~\eqref{torsprob},
\[
- \lapl v_\alpha = \alpha (1 - \alpha) \tau^{\alpha - 2} |\grad \tau|^2 + \alpha \tau^{\alpha - 1} \quad \mbox{in } \Omega.
\]
As~$\tau \in C^1(\overline{\Omega})$, we infer that $- \lapl v_\alpha \le C \alpha \tau^{\alpha - 2}$ in~$\Omega$, for some constant~$C \ge 1$ depending only on~$\Omega$. On the other hand, using again Hopf's lemma, we get that~$|\grad \tau| \ge C^{-1}$ in~$\Gamma_\varepsilon = \{ x \in \Omega : \delta(x) < \varepsilon \}$ for some~$\varepsilon > 0$ small enough. Since we also clearly have that~$\tau \ge C^{-1}$ in~$\Omega \setminus \Gamma_\varepsilon$, we deduce that~$- \lapl v_\alpha \ge C^{-1} \alpha (1- \alpha) \tau^{\alpha - 2}$ in~$\Omega$. These facts combined with~\eqref{taubounds} immediately lead to~\eqref{laplvalpha}.

We now proceed to verify~\eqref{fraclapvalpha}. To this aim, we claim that
\begin{equation} \label{D2valphaest}
\| D^2 v_\alpha \|_{L^\infty(B_{\delta(x)/4}(x))} \le C_\star \delta(x)^{\alpha - 2} \quad \mbox{for every } x \in \Omega,
\end{equation}
for some constant~$C_\star > 0$ depending only on~$N$ and~$\Omega$. This estimate is a simple consequence of the Schauder theory. Indeed, by, \textit{e.g.},~\cite{gt}*{Theorem~4.6},~\eqref{torsprob}, and the upper bound in~\eqref{taubounds}, we have
\[
\| D^2 \tau \|_{L^\infty(B_{\delta(x)/4}(x))} \le C \Big( \delta(x)^{-2} \| \tau \|_{L^\infty(B_{\delta(x)/2}(x))} + \| 1 \|_{L^\infty(B_{\delta(x)/2}(x))} \Big) \le C \delta(x)^{- 1}
\]
for some constant~$C > 0$ depending only on~$N$ and~$\Omega$. Taking advantage of this, the~$C^1(\overline{\Omega})$ regularity of~$\tau$, and~\eqref{vijcomp}, one easily deduces~\eqref{D2valphaest}. To establish~\eqref{fraclapvalpha}, we take a point~$x \in \Omega$ and write
\[
c_{N, \s}^{-1} \Ds v_\alpha(x) = \PV\int_{B_{\delta(x)/4}(x)}\frac{v_\alpha(x)-v_\alpha(y)}{{|x-y|}^{N+2\s}} \, dy+
\int_{\R^N\setminus B_{\delta(x)/4}(x)}\frac{v_\alpha(x)-v_\alpha(y)}{{|x-y|}^{N+2\s}} \, dy.
\]
On the one hand, by~\eqref{D2valphaest}, we have
\begin{multline*}
\bigg| \PV\int_{B_{\delta(x)/4}(x)}\frac{v_\alpha(x)-v_\alpha(y)}{{|x-y|}^{N+2\s}} \, dy \, \bigg| 
= \bigg| \int_{B_{\delta(x)/4}(x)}\frac{v_\alpha(x)-v_\alpha(y)+\grad v_\alpha(x)\cdot(y-x)}{{|x-y|}^{N+2\s}} \, dy \, \bigg| \\
\le  \| D^2 v_\alpha \|_{L^\infty(B_{\delta(x) / 4}(x))} |\partial B_1| \int_0^{\delta(x) / 4} t^{1 - 2 \s} \, dt 
\le C \delta(x)^{\alpha - 2 \s}.
\end{multline*}
On the other hand, using~\eqref{valpha-bounds}, we find that
\[
\int_{\R^N\setminus B_{\delta(x)/4}(x)}\frac{v_\alpha(x)-v_\alpha(y)}{{|x-y|}^{N+2\s}} \, dy \le |\partial B_1| \, v_\alpha(x) \int_{\delta(x) / 4}^{+\infty} \frac{dt}{t^{1 + 2 \s}} \le C \delta(x)^{\alpha - 2 \s}
\]
and
\[
\int_{\R^N\setminus B_{\delta(x)/4}(x)}\frac{v_\alpha(x)-v_\alpha(y)}{{|x-y|}^{N+2\s}} \, dy \ge - C \int_{\Omega \setminus B_{\delta(x)/4}(x)}\frac{\delta(y)^\alpha}{{|x-y|}^{N+2\s}} \, dy \ge - C \diam(\Omega)^\alpha \delta(x)^{- 2 \s}.
\]
Putting together the last four formulas, we arrive at~\eqref{fraclapvalpha}.

We are left to prove~\eqref{fraclapv0}. The right-hand inequality is straightforward, as
\[
c_{N, \s}^{-1} \Ds v_0(x) = \int_{\R^N\setminus\Omega}\frac{dy}{{|x-y|}^{N+2\s}}\le
\int_{\R^N\setminus B_{\delta(x)}(x)}\frac{dy}{{|x - y|}^{N+2\s}} \le C \delta(x)^{- 2\s}
\]
for every~$x \in \Omega$. To check that the left-hand one holds as well, we first observe that it suffices to establish it at points in~$\Gamma_\varepsilon = \{ x \in \Omega : \delta(x) < \varepsilon \}$, for some~$\varepsilon > 0$ arbitrarily small but depending at most on~$N$,~$\Omega$, and~$\s$. Let~$x \in \Gamma_\varepsilon$ and denote by~$z_x \in \partial \Omega$ a point for which~$\delta(x) = |x - z_x|$. By the~$C^2$ regularity of~$\partial \Omega$, there exists an exterior tangent ball~$B_{\delta(x)}(w_x)$ to~$\Omega$ at~$z_x$, provided~$\varepsilon$ is small enough (in dependence of~$\Omega$ only). By virtue of this, we compute
\[
c_{N, \s}^{-1} \Ds v_0(x) = \int_{\R^N\setminus\Omega}\frac{dy}{{|x-y|}^{N+2\s}} \ge \int_{B_{\delta(x)}(w_x)} \frac{dy}{{|x-y|}^{N+2\s}} \ge C^{-1} \delta(x)^{- 2 \s},
\]
and the lower bound in~\eqref{fraclapv0} follows. The proof of the lemma is thus complete.
\end{proof}

A simple application of the barriers constructed in Lemma~\ref{lem:valpha} is the following result, which provides estimates on how fast solutions of the Poisson equation attain their data in the presence of a right-hand side that blows up at the boundary. Half of this result is included in~\cite{gt}*{Theorem~4.9} when the domain is a ball and in \cite{gt}*{Problem~4.6} for the general case.

\begin{lemma} \label{4.9gtinOmegalem}
Let~$\alpha \in (0, 1)$,~$\Omega \subseteq \R^N$ be a bounded domain with boundary of class~$C^2$, and~$f$ be such that~$\delta^{2 - \alpha} f \in L^\infty(\Omega)$. Let~$u \in C^2(\Omega) \cap C^0(\overline{\Omega})$ be a solution of
\begin{equation*} 
\left\lbrace\begin{aligned}
- \lapl u & = f & & \hbox{in } \Omega,\\
u & = 0 & & \hbox{on } \partial \Omega.
\end{aligned}\right.
\end{equation*}
Then,
\begin{equation} \label{udeltaalphaest}
\| \delta^{-\alpha} u \|_{L^\infty(\Omega)} \le C \alpha^{-1} (1 - \alpha)^{-1} \| \delta^{2 - \alpha} f \|_{L^\infty(\Omega)},
\end{equation}
for some constant~$C \ge 1$ depending only on~$\Omega$. Furthermore, if~$f \ge 0$ in~$\Omega$, then
\begin{equation} \label{udeltaalphaestbelow}
\inf_\Omega \big( \delta^{- \alpha} u \big) \ge C^{-1} \alpha^{-1} \inf_\Omega \big( \delta^{2 - \alpha} f \big).
\end{equation}
\end{lemma}
\begin{proof}
Let~$v_\alpha$ and~$C_2$ be as in Lemma~\ref{lem:valpha}. By the left-hand inequality in~\eqref{laplvalpha}, the function
\[
\overline{v} := \alpha^{-1} (1 - \alpha)^{-1} C_2 \| \delta^{2 - \alpha} f \|_{L^\infty(\Omega)} v_\alpha
\]
satisfies
\[
\left \lbrace
\begin{aligned}
- \lapl \overline{v} & \ge \| \delta^{2 - \alpha} f \|_{L^\infty(\Omega)} \delta^{\alpha - 2} & & \hbox{in } \Omega,\\
\overline{v} & = 0 & & \hbox{on } \partial \Omega.
\end{aligned}
\right.
\]
Thus, by the weak maximum principle (applied with~$\overline{v}$ and~$- \overline{v}$ respectively as super- and subsolution), we have that~$|u| \le \overline{v} = C_2 \alpha^{-1} (1 - \alpha)^{-1} \| \delta^{2 - \alpha} f \|_{L^\infty(\Omega)} v_\alpha$ in~$\Omega$. This and~\eqref{valpha-bounds} give~\eqref{udeltaalphaest}.

When~$f \ge 0$, estimate~\eqref{udeltaalphaestbelow} can be established again via the maximum principle, this time taking advantage of the right-hand inequality in~\eqref{laplvalpha} and using~$\underline{v} := \alpha^{-1} C_2^{-1} \inf_\Omega ( \delta^{2 - \alpha} f ) v_\alpha$ as a subsolution.
\end{proof}

To deal with solutions that blow up at the boundary, we need a different class of barriers. They are provided by the next lemma, which is essentially due to~\cite{chen-felmer}. Following~\cite{chen-felmer}*{Section~3}, we define, for~$\beta \in (-1, 0)$,
\begin{equation} \label{Vbetadef}
V_\beta(x) := \begin{cases}
\eta(x) & \quad \mbox{for } x \in \Omega \setminus \Gamma_{\delta_0},\\
\delta(x)^{\beta} & \quad \mbox{for } x \in \Gamma_{\delta_0},\\
0 & \quad \mbox{for } x \in \R^N \setminus \Omega,
\end{cases}
\end{equation}
where~$\Gamma_t = \{ x \in \Omega : \delta(x) < t \}$, the parameter~$\delta_0 > 0$ is sufficiently small to have that~$\delta \in C^2(\overline{\Gamma_{\delta_0}})$, and~$\eta$ is any positive function for which~$V_\beta$ is of class~$C^2$ in~$\Omega$.

\begin{lemma} \label{Vbetalem}
Let~$\beta \in (-1 + \s, 0)$ and~$\Omega \subseteq \R^N$ be a bounded domain with boundary of class~$C^2$. Then, there exist two constants~$\delta_1 \in (0, \delta_0]$ and~$C_\sharp \ge 1$, depending only on~$N$,~$\Omega$,~$\s$, and~$\beta$, such that
\begin{align}
\label{laplVbeta}
C_\sharp^{-1} \, \delta^{\beta - 2} \le \lapl V_\beta & \le C_\sharp \, \delta^{\beta - 2} \quad \hspace{5pt} \mbox{in } \Gamma_{\delta_1},\\
\label{fraclapVbeta}
C_\sharp^{-1} \, \delta^{\beta - 2 \s} \le \Ds V_\beta & \le C_\sharp \, \delta^{\beta - 2 \s} \quad \mbox{in } \Gamma_{\delta_1}.
\end{align}
\end{lemma}
\begin{proof}
The inequalities in~\eqref{laplVbeta} are straightforward. Indeed, a simple computation using that~$|\grad \delta | = 1$ yields
\[
\lapl V_\beta = |\beta| \delta^{\beta - 2} ( 1 - \beta - \delta \lapl \delta ) \quad \mbox{in } \Gamma_{\delta_0}.
\]
Hence,~\eqref{laplVbeta} follows from the~$C^2$ regularity of~$\delta$ in~$\Gamma_{\delta_0}$ and taking~$\delta_1$ suitably small.

On the other hand,~\eqref{fraclapVbeta} is the content of~\cite{chen-felmer}*{Proposition~3.2$\,(ii)$}, once one realizes that the quantity labeled as~$\tau_0(\alpha)$ in~\cite{chen-felmer} is equal to~$-1 + \alpha$. This has already been observed in~\cite{a2}*{Remark~3.1} and is a consequence of the fact that the function~$x_+^{-1 + \alpha}$ is~$\alpha$-harmonic in~$(0, +\infty)$---the function appearing in~\cite{chen-felmer}*{formula~(1.13)} is equal, up to an irrelevant factor, to the~$\alpha$-Laplacian of~$x_+^{\tau}$ evaluated at~$x=1$. The~$\alpha$-harmonicity of~$x_+^{- 1 + \alpha}$ in~$(0, +\infty)$ can be verified in several ways---see,~\textit{e.g.},~\cite{ADFJS}*{Lemma~4.1} for a proof based on the computations of~\cite{dyda}.
\end{proof}

\section{Comparison principles}\label{sec:CP}

In this section, we prove the weak comparison principle of Proposition~\ref{WCPprop} and deduce from it some estimates on the supremum of subsolutions of~\eqref{prob}.

\begin{proof}[Proof of Proposition~\ref{WCPprop}]
Assume first that both inequalities in~\eqref{wover>wunderonboundary} and on the first line of~\eqref{WCPine} are strict, \textit{i.e.}, that~$\underline{w}$ and~$\overline{w}$ satisfy
\begin{equation} \label{WCPineaux}
\left\lbrace\begin{aligned}
- \lapl \underline{w} + \Phi(\, \cdot \,, \Ds \underline{w}) 
&< - \lapl \overline{w} + \Phi(\, \cdot \,, \Ds \overline{w}) & & \mbox{in } \Omega\\
\underline{w} &\le \overline{w} & & \mbox{in } \R^N \setminus \overline{\Omega}\\
\limsup_{\Omega \ni x \rightarrow x_0} \underline{w}(x) &< \liminf_{\Omega \ni x \rightarrow x_0} \overline{w}(x) & & \mbox{for all } x_0 \in \partial \Omega.
\end{aligned}\right.
\end{equation}
Let~$w := \underline{w} - \overline{w}$ and
\[
M := \sup_{\Omega} w.
\]
We claim that
\begin{equation} \label{Mle0claim}
M \le 0.
\end{equation}
Of course, if~\eqref{Mle0claim} is valid, then we are done. Therefore, we argue by contradiction and suppose that~$M > 0$.

By the continuity of~$\underline{w}$ and~$\overline{w}$ inside~$\Omega$ and the strict inequality on the third line of~\eqref{WCPineaux}, there exists a point~$x_M \in \Omega$ at which~$w(x_M) = M$. As~$w \le 0$ outside of~$\overline{\Omega}$, we infer that
\[
w(x_M) = M = \max_{\R^N} w.
\]
Accordingly,
\[
- \lapl w(x_M) \ge 0 \quad \mbox{and} \quad \Ds w(x_M) \ge 0,
\]
that is,
\[
- \lapl \underline{w}(x_M) \ge - \lapl \overline{w}(x_M) \quad \mbox{and} \quad \Ds \underline{w}(x_M) \ge \Ds \overline{w}(x_M).
\]
In view of this and the monotonicity of~$\Phi(x_M, \cdot)$, we obtain that
\[
- \lapl \underline{w}(x_M) + \Phi(x_M, \Ds \underline{w}(x_M)) \ge - \lapl \overline{w}(x_M) + \Phi(x_M, \Ds \overline{w}(x_M)),
\]
in contradiction with the first inequality in~\eqref{WCPineaux}. Thus,~\eqref{Mle0claim} holds true and the lemma is proved when~\eqref{WCPineaux} is in force.

Suppose now that~$\underline{w}$ and~$\overline{w}$ satisfy the weaker hypotheses~\eqref{WCPine} and~\eqref{wover>wunderonboundary}. Let~$R > 0$ be large enough to have~$\overline\Omega  \subseteq B_R$ and consider the function~$u_\s(x) := (R^2 - |x|^2)_+^\s$. By Lemma~\ref{easybarlem}, we know that
\begin{equation} \label{w2sprop}
\Ds u_\s > 0 \quad \mbox{and} \quad - \lapl u_\s \ge 2 \s N R^{2 \s - 2} > 0 \qquad \mbox{in } B_R.
\end{equation}
Consequently, letting~$\overline{w}_\varepsilon := \overline w + \varepsilon u_\s$ for any small~$\varepsilon > 0$ and using again the monotonicity of~$\Phi$ with respect to the second variable, we see that
\[
\left\lbrace\begin{aligned}
- \lapl \underline{w} + \Phi(\, \cdot \,, \Ds \underline{w}) 
&< - \lapl \overline{w}_\varepsilon + \Phi(\, \cdot \,, \Ds \overline{w}_\varepsilon) 
& & \mbox{in } \Omega \\
\underline{w} &\le \overline{w}_\varepsilon 
& & \mbox{in } \R^N \setminus \Omega\\
\limsup_{\Omega \ni x \rightarrow x_0} \underline{w}(x) &< \liminf_{\Omega \ni x \rightarrow x_0} \overline{w}_\varepsilon(x) & & \mbox{for all } x_0 \in \partial \Omega.
\end{aligned}\right.
\]
By what we established before,~$\underline{w} \le \overline{w}_\varepsilon \le \overline{w} + \varepsilon R^{2 \s}$ in the whole~$\R^N$. The conclusion of the lemma now follows by letting~$\varepsilon \downarrow 0$.
\end{proof}

As applications of Proposition~\ref{WCPprop}, we have two results providing upper bounds on the supremum of subsolutions of~\eqref{prob}. Of course, from these one may easily deduce the corresponding lower bounds for supersolutions and two-sided bounds for solutions.

First, we suppose the right-hand side~$f$ in~\eqref{prob} to be a bounded function. In this case, it suffices to apply Proposition~\ref{WCPprop} in conjunction with the barrier of Lemma~\ref{easybarlem}.

\begin{corollary} 
Let~$\Omega \subseteq \R^N$ be a bounded open set and~$\Phi: \Omega \times \R \to \R$ be a Carath\'eodory function with~$\Phi(x, \cdot)$ non-decreasing and~$\Phi(x, 0) \ge 0$ for a.e.~$x \in \Omega$. Let~$f \in L^\infty(\Omega)$,~$g \in L^\infty(\partial \Omega)$, $h \in L^\infty(\R^N \setminus \overline{\Omega})$, and~$\underline{w} \in L^\infty(\R^N) \cap C^2(\Omega) \cap C^0(\overline{\Omega})$ be such that
\begin{equation} \label{underwsub}
\left\lbrace\begin{aligned}
- \lapl \underline{w} + \Phi(\, \cdot \,, \Ds \underline{w}) 
& \le f & & \mbox{in } \Omega \\
\underline{w} &\le g & & \mbox{on } \partial \Omega \\\underline{w} &\le h & & \mbox{in } \R^N \setminus \overline\Omega.
\end{aligned}\right.
\end{equation}
Then,
\[
\sup_{\Omega} \underline{w} \le \sup_{\partial \Omega} g_+ + \sup_{\R^N \setminus \overline{\Omega}} h_+ + C \, \diam(\Omega)^2 \, \sup_{\Omega} f_+,
\]
for some constant~$C > 0$ depending only on~$N$ and~$\s$.
\end{corollary}
\begin{proof}
Write~$R := \diam(\Omega)$ and pick~$x_0 \in \R^N$ in such a way that~$\overline\Omega \subseteq B_R(x_0)$. Without loss of generality, we may assume that~$x_0 = 0$. Similarly to what we did at the end of the proof of Proposition~\ref{WCPprop}, we consider the function~$u_\s(x) := (R^2 - |x|^2)^\s_+$, which satisfies~\eqref{w2sprop}. Hence, the function
\[
\overline{w}(x) := \sup_{\partial \Omega} g_+ + \sup_{\R^N \setminus \overline{\Omega}} h_+ + \frac{R^{2(1 - \s)} \sup_{\Omega} f_+}{2 N \s} \, u_\s(x),
\qquad x\in\R^N,
\]
is such that
\begin{equation} \label{wbarsuper}
\left\lbrace\begin{aligned}
- \lapl \overline{w} + \Phi(\, \cdot \,, \Ds \overline{w}) & \ge f & & \mbox{in } \Omega \\
\overline{w} & \ge g & & \mbox{on } \partial \Omega \\
\overline{w} & \ge h & & \mbox{in } \R^N \setminus \Omega.
\end{aligned} \right.
\end{equation}
The conclusion now follows by Proposition~\ref{WCPprop}.
\end{proof}

By combining Proposition~\ref{WCPprop} with Lemma~\ref{lem:valpha}, we may tackle the case when~$f$ is merely in~$L^\infty_{\loc}(\Omega)$ and blows up at the boundary of~$\Omega$ at a strictly slower rate than the square of the inverse distance function.

\begin{corollary} \label{WCPcor2}
Let~$\Omega \subseteq \R^N$ be a bounded open set with boundary of class~$C^2$ and~$\Phi: \Omega \times \R \to \R$ be a Carath\'eodory function with~$\Phi(x, \cdot)$ non-decreasing and~$\Phi(x, 0) \ge 0$ for a.e.~$x \in \Omega$. Let~$f$ be such that~$\delta^{2-\alpha}f\in L^\infty(\Omega)$ for some~$0 < \alpha \le \overline{\alpha} < 1$,~$g \in L^\infty(\partial \Omega)$, and~$h \in L^\infty(\R^N \setminus \overline{\Omega})$. Let~$\underline{w} \in L^\infty(\R^N) \cap C^2(\Omega) \cap C^0(\overline{\Omega})$ be such that~\eqref{underwsub} holds true. Then,
\begin{align}\label{45455}
\sup_{\Omega} \underline{w} \leq \sup_{\partial \Omega} g_+ + \sup_{\R^N \setminus \overline{\Omega}} h_+ + C\alpha^{-1} \sup_\Omega \big( \delta^{2-\alpha}f_+ \big)
\end{align}
for some constant~$C > 0$ depending only on~$N$,~$\Omega$,~$\s$, and~$\overline{\alpha}$.
\end{corollary}
\begin{proof}
For~$M > 0$, define the non-negative function
\[
w(x):=\alpha^{-1} \sup_{\Omega} \big( \delta^{2-\alpha}f_+ \big) \big( M v_0(x) + C_2 (1 - \alpha)^{-1} v_\alpha(x) \big),
\qquad x \in \R^N,
\]
where~$v_0$,~$v_\alpha$, and~$C_2$ are as in Lemma~\ref{lem:valpha}. By estimates~\eqref{laplvalpha}-\eqref{fraclapv0}, in~$\Omega$ we then have
\begin{align*}
-\lapl{w} & \ge \sup_{\Omega} \big( \delta^{2-\alpha}f_+ \big) \delta^{\alpha-2}, \\
\Ds{w} & \ge \alpha^{-1} \sup_{\Omega} \big( \delta^{2-\alpha}f_+ \big)
\big( C_3^{-1} M -C_2^2 (1 - \alpha)^{-1} \, \diam(\Omega)^\alpha \big) \delta^{-2\s} \ge 0,
\end{align*}
provided~$M$ is large enough, in dependence of~$N$,~$\Omega$,~$\s$, and~$\overline{\alpha}$ only. Thus,
\[
- \lapl w +\Phi( \,\cdot\,, \Ds w) \ge f
\qquad\text{in }\Omega.
\] 
This yields that~$\overline{w}:=w + \sup_{\partial \Omega} g_+ + \sup_{\R^N \setminus \overline{\Omega}} h_+$ satisfies~\eqref{wbarsuper}. From Proposition~\ref{WCPprop} it follows that~$\underline{w}\leq\overline{w}$. Estimate~\eqref{45455} is then a consequence of Lemma~\ref{lem:valpha}.
\end{proof}

\section{Existence. Proof of Theorem~\ref{thm:ex}} \label{sec:ex}

We present here the proof of Theorem~\ref{thm:ex} under the notational conventions explained in Paragraph~\ref{notations}.
For readability purposes, we split the proof into four intermediate steps:
\begin{enumerate}[label=Step \arabic*),leftmargin=5em]
	\item First, we reduce~\eqref{prob} to an equivalent problem having vanishing boundary and exterior data.
	\item The so-obtained Dirichlet problem will contain singular terms originating from the lack of smoothness of the fractional Laplacian at the boundary 
and we circumvent this issue by solving
a family of regularized problems by cutting the singularities off.
	\item We then obtain uniform estimates on the solutions to these regularized problems; to get stronger estimates, we use the family of weighted Sobolev spaces introduced in Section~\ref{sec:lototsky}, wherein all necessary notation can be found.
	\item Finally, the estimates enable us to conclude that the solutions to the regularized problems accumulate at a solution of the original one. Its uniqueness then immediately follows from Proposition~\ref{WCPprop}.
\end{enumerate}

\subsection{Reduction to homogeneous data}

Let~$\bar{g}$ be the harmonic extension of~$g$ inside~$\Omega$, \textit{i.e.},~$\bar{g} \in C^2(\Omega) \cap C^0(\overline{\Omega})$ is the unique solution of the Dirichlet problem
\[
\left\lbrace\begin{aligned}
- \lapl \bar{g} & = 0 & & \hbox{in } \Omega \\
\bar{g} & = g & & \hbox{on } \partial \Omega.
\end{aligned}\right.
\]
We define
\[
\psi := \chi_{\overline{\Omega}} \bar{g} + \chi_{\R^N \setminus \overline{\Omega}} h = 
\left\lbrace\begin{aligned}
\bar{g} &&& \mbox{in } \overline{\Omega} \\
h &&& \mbox{in } \R^N \setminus \overline{\Omega}.
\end{aligned} \right.
\]
Notice that~$\psi \in L^\infty(\R^N) \cap C^2(\Omega) \cap C^0(\overline{\Omega})$.

Letting~$v := u - \psi$, it is clear that~\eqref{prob} is equivalent to the problem\footnote{Let us stress here that~$\R^N\setminus\Omega=\partial\Omega\cup(\R^N\setminus\overline\Omega)$.}
\begin{equation}\label{probv}
\left\lbrace\begin{aligned}
-\lapl v+ P[v] &= f & & \hbox{in }\Omega \\
v &= 0 & & \hbox{in }\R^N\setminus\Omega,
\end{aligned}\right.
\end{equation}
with
\begin{align*}
P[v] := \Big\{ \big|\Ds v + \Ds \psi \big|^{p-1} \big(\Ds v + \Ds \psi \big) \Big\} \Big|_\Omega.
\end{align*}
Observe that both~$\Ds \psi$ and~$f$ are locally bounded and H\"older continuous functions in~$\Omega$. However, they may in general blow up at the boundary of~$\Omega$. Indeed, we have that
\begin{equation} \label{pointDspsibound}
|\Ds \psi(x)| \le C \Big( \| g \|_{L^\infty(\partial \Omega)} + \| h \|_{L^\infty(\R^N \setminus \Omega)} \Big) \delta(x)^{- 2 \s} \qquad \mbox{for all } x \in \Omega.
\end{equation}
To see this, on the one hand, by the maximum principle and the classical Schauder theory (\textit{e.g.},~\cite{gt}*{Theorem~4.6}), one gets that
\[
\| \bar{g} \|_{L^\infty(\Omega)} + \delta(x)^2 \| D^2 \bar{g} \|_{L^\infty(B_{\delta(x)/4}(x))} \le C \| g \|_{L^\infty(\partial \Omega)} \qquad \mbox{for all } x \in \Omega.
\]
Consequently, arguing as we did to get estimate~\eqref{fraclapvalpha} in Lemma~\ref{lem:valpha}, we find that
\begin{equation*}
|\Ds (\chi_{\overline{\Omega}} \bar{g})(x)| \le C \| g \|_{L^\infty(\partial \Omega)} \delta(x)^{- 2 \s},
\end{equation*}
for all~$x \in \Omega$. On the other hand, by computing directly,
\begin{equation*} 
|\Ds (\chi_{\R^N \setminus \overline{\Omega}} h)(x)| \le c_{N, \s} \int_{\R^N\setminus B_{\delta(x)}(x)} \frac{\chi_{\R^N \setminus \overline{\Omega}}(y) |h(y)|}{|x - y|^{N + 2 \s}} \, dy \le C \| h \|_{L^\infty(\R^N \setminus \overline{\Omega})} \delta(x)^{- 2 \s},
\end{equation*}
for all~$x \in \Omega$. The combination of the last two estimates leads to~\eqref{pointDspsibound}.

In light of this diverging behaviours, to solve~\eqref{probv} it is convenient to consider a family of suitably regularized problems. This will be the content of the next subsection.

\subsection{Approximating problems}

For any large integer~$j$, consider the open set
\[
\Omega_j := \big\{ x \in \Omega : \delta(x) > 2^{-j} \big\}.
\]
Then, let~$\eta_j \in C^\infty_c(\R^N)$ be a cut-off function satisfying~$0 \le \eta_j \le 1$ in~$\R^N$,~$\supp(\eta_j) \subseteq \Omega_{j}$,~$\eta_j \equiv 1$ in~$\Omega_{j - 1}$, and~$|\grad \eta_j| \le C_j$ in~$\R^N$. We take into account the auxiliary problem
\begin{equation}\label{probvj}
\left\lbrace\begin{aligned}
-\lapl v+ P_j[v] &= \eta_jf & & \hbox{in }\Omega \\
v &= 0 & & \hbox{in }\R^N\setminus\Omega,
\end{aligned}\right.
\end{equation}
with~$P_j[v] := \eta_j P[v]$. To find a solution of~\eqref{probvj}, we will look at it as a fixed-point problem.

Let~$\beta \in (2 \s, 2) \setminus \{ 1 \}$ to be chosen later, in dependence of~$\s$ and~$p$ only, and consider the Banach space
\[
\X := \Big\{ w \in C^0(\R^N) \cap C^\beta(\overline{\Omega}) : w = 0 \mbox{ in } \R^N \setminus \Omega \Big\},
\]
endowed with the norm~$\| w \|_{\X} := \| w \|_{C^\beta(\overline{\Omega})}$.

First, we claim that~$P_j: \X \to L^\infty(\Omega)$ is a continuous mapping and that
\begin{equation} \label{Pjbound}
\| P_j[w] \|_{L^\infty(\Omega)} \le \overline{C}_j \big( 1 + \| w \|_\X^p \big) \quad \mbox{for every } w \in \X.
\end{equation}
The continuity easily follows from the fact that~$\supp(\eta_j) \subseteq \Omega_j$ and the estimate
\begin{multline*}
\big|\Ds w(x)\big| = \frac{c_{N,\s}}{2} \bigg| \int_{\R^N} \frac{2 w(x) - w(x + z) - w(x - z)}{|z|^{N + 2 \s}} \, dz \, \bigg| \leq \\
\le C \bigg( \int_{B_{2^{- j - 1}}} \frac{ [w]_{C^\beta(B_{2^{- j - 1}}(x))}}{|z|^{N + 2 \s - \beta}} \, dz + \int_{\R^N \setminus B_{2^{- j - 1}}} \frac{ \| w \|_{L^\infty(\R^N \setminus B_{2^{- j - 1}}(x))}}{|z|^{N + 2 \s}} \, dz \bigg) 
\le C_j \| w \|_{\X},
\end{multline*}
which holds true for every~$w \in \X$ and~$x \in \Omega_j$. From this and~\eqref{pointDspsibound}, we also infer that
\begin{multline*}
\big\| P_j[w] \big\|_{L^\infty(\Omega)} 
\le C \Big( \| \Ds w \|_{L^\infty(\Omega_j)}^p + \| \Ds \psi \|_{L^\infty(\Omega_j)}^p \Big) \leq \\
\le C_j \Big( \| w \|_{\X}^p + \| g \|_{L^\infty(\partial \Omega)}^p + \| h \|_{L^\infty(\R^N \setminus \Omega)}^p \Big),
\end{multline*}
which gives~\eqref{Pjbound}.

Denote now by~$(- \lapl)^{-1}$ the inverse of the Dirichlet Laplacian in~$\Omega$, \textit{i.e.}, let~$(-\lapl)^{-1} F$ be the only solution~$\phi$ of the problem
\[
\left\lbrace\begin{aligned}
-\lapl \phi &= F & & \hbox{in }\Omega \\
\phi &= 0 & & \hbox{on }\partial \Omega.
\end{aligned}\right.
\]
By the classical Calder\'on-Zygmund theory and the Sobolev embedding, the operator~$(-\lapl)^{-1}$ maps~$L^\infty(\Omega)$ into~$W^{1, q}_0(\Omega) \cap W^{2, q}(\Omega) \cap C^{\gamma}(\overline\Omega)$, for every~$q \in (1, +\infty)$ and~$\gamma \in (0, 2)$. Also,
\begin{equation} \label{-lapl-1cont}
\| (-\lapl)^{-1} F \|_{C^{\gamma}(\overline\Omega)} \le C_{\gamma} \| F \|_{L^\infty(\Omega)}.
\end{equation}

Pick now any~$q \in [1, +\infty)$ and~$\gamma \in (\beta, 2)$. Then, the standard inclusion~$\iota: C^{\gamma}(\overline\Omega) \to C^\beta(\overline\Omega)$ is compact. Hence, the mapping~$T_j: \X \to \X$ defined by
\[
T_j[w] := 
\begin{cases}
\iota \big( (-\lapl)^{-1} ( \eta_jf - P_j[w] ) \big) & \hbox{in } \Omega \\
0 & \hbox{in } \R^N \setminus \Omega,
\end{cases}
\]
for all~$w \in \X$, is also compact. We stress that~$T_j[w]$ defines a continuous function in~$\R^N$---and is thus an element of~$\X$---since its restriction to~$\Omega$ belongs to~$W_0^{1, q}(\Omega) \cap C^\beta(\overline\Omega)$.

Notice then that~$v \in \X \cap W^{2, q}(\Omega)$ is a solution of~\eqref{probvj} if and only if it is a fixed point of the map~$T_j$. Since~$T_j$ is compact, we can show the existence of a fixed point using the Leray-Schauder Theorem (see, \textit{e.g.},~\cite{gt}*{Theorem~11.3}), provided we check that
\begin{equation} \label{LSunivbound}
\| v \|_\X \le \overline{C}_j \quad \mbox{for every } v \in \X \mbox{ such that } v = \lambda T_j[v] \mbox{ for some } \lambda \in [0, 1].
\end{equation}

To see this, note that if~$v \in \X$ satisfies~$v = \lambda T_j[v]$, then~$v$ is a~$C^0(\R^N) \cap W^{2, q}(\Omega)$ solution of
\[
\left\lbrace\begin{aligned}
-\lapl v + \lambda P_j[v] &= \lambda \eta_jf & & \hbox{in }\Omega, \\
v &= 0 & & \hbox{in }\R^N\setminus\Omega.
\end{aligned}\right.
\]
Then, by standard elliptic regularity,~$v$ is actually of class~$C^2$ in~$\Omega$, and therefore the function~$u := v + \psi$ is a~$L^\infty(\R^N) \cap C^2(\Omega) \cap C^0(\overline{\Omega})$ solution of
\[
\left\lbrace\begin{aligned}
-\lapl u + \lambda \eta_j \big|\Ds u\big|^{p-1} \Ds u  &= \lambda\eta_j f & & \hbox{in }\Omega \\
u &= g & & \hbox{on } \partial \Omega \\
u &= h & & \hbox{in }\R^N\setminus\Omega.
\end{aligned}\right.
\]
Hence, by the comparison principle of Corollary~\ref{WCPcor2}, we infer that~$\| u \|_{L^\infty(\Omega)}$ is universally bounded, and thus
\begin{equation} \label{vunivbound}
\| v \|_{L^\infty(\Omega)} \le \overline{C}.
\end{equation}
Knowing this, we may proceed to show the validity of~\eqref{LSunivbound}. First, we remark that
\[
\| \varphi \|_{C^{\alpha_1}(\overline\Omega)} \le C \| \varphi \|_{L^\infty(\Omega)}^{1 - \alpha_1 / \alpha_2} \| \varphi \|_{C^{\alpha_2}(\overline\Omega)}^{\alpha_1 / \alpha_2} \quad \mbox{for all } \varphi \in C^{\alpha_2}(\overline\Omega) \mbox{ and } 0 < \alpha_1 < \alpha_2 < 2.
\]
See, \textit{e.g.},~\cite{lunardi_analsemi}*{Proposition~1.1.3$\,(iii)$}. Thanks to this,~\eqref{vunivbound},~\eqref{-lapl-1cont}, and~\eqref{Pjbound}, we compute
\begin{align*}
\| v \|_\X & = \| v \|_{C^\beta(\overline\Omega)} \le C \| v \|_{L^\infty(\Omega)}^{1 - \beta / \gamma} \| v \|_{C^\gamma(\overline\Omega)}^{\beta / \gamma} 
\le \overline{C} \lambda^{\beta / \gamma} \| T_j[v] \|_{C^\gamma(\overline\Omega)}^{\beta / \gamma} \\
& \le \overline{C}_\gamma \| \eta_jf - P_j[v] \|_{L^\infty(\Omega)}^{\beta / \gamma} 
\le \overline{C}_\gamma \Big( \| \eta_j f \|_{L^\infty(\Omega)}^{\beta / \gamma} + \| P_j[v] \|_{L^\infty(\Omega)}^{\beta / \gamma} \Big) 
\le \overline{C}_{\gamma, j} \Big( 1 + \| v \|_\X^{\beta p/ \gamma}  \Big).
\end{align*}
Notice that, since~$\s p < 1$, we can choose~$\beta \in (2 \s, 2) \setminus \{ 1 \}$ (close to~$2 \s$) and~$\gamma \in (\beta, 2)$ (close to~$2$) in a way that~$\beta p / \gamma < 1$. By doing this and applying the weighted Young's inequality, claim~\eqref{LSunivbound} easily follows from the above estimate.

Accordingly, we can apply the Leray-Schauder Theorem and conclude that there exists a fixed point~$v_j \in \X$ for the map~$T_j$, \textit{i.e.}, a solution~$v_j \in C^0(\R^N) \cap C^2(\Omega)$ of problem~\eqref{probvj}. Also,
\begin{equation} \label{vjunivbound}
\| v_j \|_{L^\infty(\Omega)} \le \overline{C},
\end{equation}
as a consequence of~\eqref{vunivbound}.

\subsection{Uniform estimates} \label{subsub:uni}

We now want to let~$j \uparrow \infty$ and show that the limit of the~$v_j$'s is a solution of~\eqref{probv}. In order to do this, we need estimates for~$v_j$ that do not depend on~$j$. Note that we already know that each~$v_j$ satisfies the uniform~$L^\infty$ bound~\eqref{vjunivbound}.

Let~$q \in (1, +\infty)$. As~$v_j$ is a solution of~\eqref{probvj}, by~Proposition~\ref{lotest}, Lemma~\ref{DsleH2slem} 
(applied here with~$p q$ in place of~$p$,~$r = pq$,~$\theta = N + 2 q > N + 2 \s p q$, and~$\varepsilon = 2 (1 - \s p) q > 0$), and estimate~\eqref{vjunivbound}, we have
\begin{equation} \label{123132132132}
\begin{aligned}
\| v_j \|_{L^{2, q}_N(\Omega)} 
& \le C_q \Big( \| \eta_j f - P_j[v_j] \|_{L^q_{N + 2 q}(\Omega)} + \| v_j \|_{L^q_N(\Omega)} \Big) \\
& \le C_q \Big( \| \Ds v_j \|_{L^{p q}_{N + 2 q}(\Omega)}^p + \| \Ds \psi \|_{L^{p q}_{N + 2 q}(\Omega)}^p + \| f \|_{L^q_{N + 2 q}(\Omega)} + \| v_j \|_{L^q(\Omega)} \Big) \\
& \le \overline{C}_q \Big( 1 + \| v_j \|_{L^{2 \s, p q}_{N}(\Omega)}^p \Big).
\end{aligned}
\end{equation}
Notice that, to get the last inequality, we also took advantage of the fact that, thanks to~\eqref{pointDspsibound} and~\eqref{hypo-data},
\begin{multline*}
\| \Ds \psi \|_{L^{p q}_{N + 2 q}(\Omega)}^{p q} + \| f \|_{L^q_{N + 2 q}(\Omega)}^q \le C_q \int_{\Omega} \Big( \big|\Ds \psi(x)\big|^{p q} + |f(x)|^q \Big) \delta(x)^{2 q} \, dx \\
\le \overline{C}_q \int_\Omega \Big( \delta(x)^{2 q (1 - \s p)} + \delta(x)^{\alpha q} \Big) \, dx \le \overline{C}_q.
\end{multline*}
The interpolation inequality of, say,~\cite{lunardi}*{Corollary~2.1.8} along with the representation of Proposition~\ref{lotprops}$.v$ for the space~$L^{2 \s, p q}_N(\Omega)$ and again~\eqref{vjunivbound} then give that
\[
\| v_j \|_{L^{2 \s, p q}_N(\Omega)} \le C_q \| v_j \|_{L^{\frac{(1 - \s) p q}{1 - \s p}}_N(\Omega)}^{1 - \s} \| v_j \|_{L^{2, q}_N(\Omega)}^\s \le \overline{C}_q \| v_j \|_{L^{2, q}_N(\Omega)}^\s.
\]
By plugging this into~\eqref{123132132132} and taking advantage of the weighted Young's inequality, we conclude that
\begin{equation} \label{vjunifH2qest}
\| v_j \|_{L^{2, q}_N(\Omega)} \le \overline{C}_q \quad \mbox{for every } j\in\N.
\end{equation}
Note that, once again, we used in a crucial way that~$\s p < 1$.

Next, we claim that, for any small~$\varepsilon \in (0, 1)$,
\begin{equation} \label{vjunifdecay}
|v_j(x)| \le \overline{C}_\varepsilon \, \delta(x)^{\min \{ 1 - \varepsilon, 2 (1 - \s p) - \varepsilon, \alpha \}} \quad \mbox{for every } x \in \Omega \mbox{ and every } j\in\N.
\end{equation}
To check this, let~$q > N$ and notice that, using~\eqref{vjunifH2qest} and the standard Morrey's inequality,
\begin{equation*} 
\begin{aligned}
\big[ \grad v_j \big]_{C^{1 - N/q}(B_{\delta(x) / 2}(x))} & \le C_q \Big( \delta(x)^{- q} \big\| \grad v_j \big\|_{L^q(B_{\delta(x)/2}(x))}^q + \big\| D^2 v_j \big\|_{L^q(B_{\delta(x) / 2}(x))}^q \Big)^{1/q} \\
& \le C_q \delta(x)^{-2} \bigg( \int_{\Omega} |\grad v_j(y)|^q \delta(y)^{q} dy + \int_{\Omega} |D^2 v_j(y)|^q \delta(y)^{2 q} dy \bigg)^{1/q} \\
& \le C_q \delta(x)^{-2} \big\| v_j \big\|_{L^{2, q}_N(\Omega)} \le \overline{C}_q \delta(x)^{-2}.
\end{aligned}
\end{equation*}
Hence, from this and~\eqref{vjunivbound}, it easily follows that
\begin{multline*}
\big|\Ds v_j(x)\big| 
\le C \bigg( \int_{B_{\delta(x)/2}} \frac{ [ \grad v_j ]_{C^{1 - N / q}(B_{\delta(x)/2}(x))}}{|z|^{N + 2 \s - 2 + N/q}} \, dz + \int_{\R^N \setminus B_{\delta(x)/2}} \frac{\| v_j \|_{L^\infty(\Omega)}}{|z|^{N + 2 \s}} \, dz \bigg) \leq \\
\le \overline{C}_q \delta(x)^{- 2 \s - N/q},
\end{multline*}
for every~$x \in \Omega$ and~$q$ large enough. Estimate~\eqref{vjunifdecay} is then a consequence of this inequality,~\eqref{pointDspsibound}, and Lemma~\ref{4.9gtinOmegalem}, recalling that~$v_j$ is a solution of~\eqref{probvj} and taking~$q = N p / \varepsilon$, with~$\varepsilon > 0$ sufficiently small.

Thanks to the uniform bounds~\eqref{vjunifH2qest} and~\eqref{vjunifdecay}, we are now able to get a limit for~$v_j$ as~$j \uparrow \infty$ and obtain a solution of~\eqref{probv}. 
Bear in mind that estimate~\eqref{vjunifH2qest} gives in particular that
\[
\| v_j \|_{W^{2, q}(\Omega_k)} \le \overline{C}_{q, k},
\]
for any~$q > N$. To see this, it is convenient to recall the equivalent representation for~$L^{2, q}_N(\Omega)$ given in Proposition~\ref{lotprops}$.iv$. Thus, 
by Morrey's inequality,
\begin{equation} \label{vjunifinCgamma}
\| v_j \|_{C^{\gamma}(\overline{\Omega_k})} \le \overline{C}_{\gamma, k},
\end{equation}
for any fixed~$\gamma \in (\max \{ 2 \s, 1 \}, 2)$ and every large integers~$j$ and~$k$.

\subsection{Passage to the limit} \label{subsub:lim}
By the compact embedding of H\"older spaces, bound~\eqref{vjunifinCgamma}, and a standard diagonal procedure,~${(v_j)}_{j\in\N}$ converges (up to a subsequence) to a function~$v$ in~$C^{\gamma}_\loc(\Omega)$. Letting~$j \uparrow \infty$ in~\eqref{vjunifdecay}, we obtain that the extension of~$v$ to~$0$ outside~$\Omega$ (that we still call~$v$) defines a continuous function on the whole~$\R^N$. By the dominated convergence theorem and the uniform~$L^\infty$ bound~\eqref{vjunivbound}, we also have that~$v_j \rightarrow v$ in~$L^1(\R^N)$.

Passing~\eqref{probvj} to the limit, one easily obtains that~$v$ is a weak solution of~$- \lapl v + P[v] = f$ in every open set compactly contained in~$\Omega$, that is
\[
\int_{\Omega} \big( \grad v \cdot \grad \varphi + P[v] \varphi \big) 
= \int_{\Omega} f \varphi \quad \mbox{for all } \varphi \in C^\infty_c(\Omega).
\]
Notice that this can be done since~$P_j[v_j]$ converges to~$P[v]$ in~$L^\infty_\loc(\Omega)$, as a consequence of the convergence of~$v_j$ to~$v$ in~$C^\gamma_\loc(\Omega)$ and~$L^1(\R^N)$. Now, since~$f$ and~$P[v]$ are both locally H\"older continuous functions in~$\Omega$ (as~$v \in C^\gamma_\loc(\Omega)$ with~$\gamma > 2 \s$), by elliptic regularity we conclude that~$v$ belongs to~$C^0(\R^N) \cap C^2(\Omega)$ and solves~\eqref{probv} pointwise. The proof of Theorem~\ref{thm:ex} is then complete.

\section{Non-existence. Proofs of Theorems~\ref{thm:non-ex} and~\ref{thm:non-ex-sourc}} \label{sec:non-ex}

In this section, we establish our two non-existence results, which are valid respectively when~$\s p \ge 1$ or when the right-hand side~$f$ blows up too rapidly at the boundary of~$\Omega$.

First, we deal with the case of vanishing right-hand side and critical or supercritical regime.

\begin{proof}[Proof of Theorem~\ref{thm:non-ex}.]
Letting~$\alpha, \varepsilon > 0$,~$m:=\max_{\partial\Omega} g > 0$,~$v_0 = \chi_{\overline{\Omega}}$, and $v_\alpha$ be as in Lemma~\ref{lem:valpha}, we consider the function
\[
\overline{w}_{\alpha, \varepsilon} := m ( v_0 - \varepsilon v_\alpha ).
\]
We claim that there exists an~$\varepsilon_0 \in (0, 1)$ small enough such that~$\overline{w}_{\alpha, \varepsilon}$ is a supersolution to problem~\eqref{prob} for any~$\alpha \in (0, \s)$ and~$\varepsilon \in (0, \varepsilon_0]$.
	
Clearly,~$\overline{w}_{\alpha, \varepsilon} \in L^\infty(\R^N) \cap C^2(\Omega) \cap C^0(\overline{\Omega})$. Moreover,
\begin{equation*}
\overline{w}_{\alpha, \varepsilon} \ge g \mbox{ on } \partial \Omega \quad \mbox{and} \quad \overline{w}_{\alpha, \varepsilon} \ge h \mbox{ in } \R^N \setminus \Omega,
\end{equation*}
thanks to the definition of~$m$ and the fact that~$h$ is non-positive. Therefore, in order to prove that~$\overline{w}_{\alpha, \varepsilon}$ is a supersolution to~\eqref{probf=0}, we only need to check that
\begin{equation} \label{walphaepssuper}
- \lapl \overline{w}_{\alpha, \varepsilon} + \big| \Ds \overline{w}_{\alpha, \varepsilon} \big|^{p - 1} \Ds \overline{w}_{\alpha, \varepsilon} \ge 0 \quad \mbox{in } \Omega,
\end{equation}
provided~$\varepsilon$ is sufficiently small, uniformly with respect to~$\alpha \in (0, \s)$.

To do this, we take into account formulas~\eqref{laplvalpha}-\eqref{fraclapv0} of Lemma~\ref{lem:valpha}, that give
\[- \lapl v_\alpha \le C_\star \delta^{\alpha-2} \quad \mbox{and} \quad \Ds v_\alpha \le C_\star \delta^{-2\s} \quad \mbox{in }\Omega,
\]
and
\[
- \lapl v_0 = 0 \quad \mbox{and} \quad \Ds v_0 \ge C_\star^{-1} \delta^{-2\s} 
\quad \mbox{in } \Omega,
\]
for some constant~$C_\star \ge 1$ depending only on~$N$,~$\s$, and~$\Omega$.
In particular,
\[
- \lapl \overline{w}_{\alpha, \varepsilon} = \varepsilon m \lapl v_\alpha \ge - C_\star m \varepsilon \, \delta^{\alpha-2} \quad \mbox{in } \Omega
\]
and
\[
\Ds \overline{w}_{\alpha, \varepsilon} = m \Big( \Ds v_0 - \varepsilon \Ds v_\alpha \Big) \ge m \big( C_\star^{-1} - C_\star \varepsilon \big) \delta^{- 2 \s} \ge \frac{m}{2 C_\star} \, \delta^{- 2 \s} \qquad \mbox{in } \Omega,
\]
provided~$\varepsilon \le (2 C_\star^2)^{-1}$. In light of these two relations, we obtain that, in~$\Omega$,
\begin{multline*}
- \lapl \overline{w}_{\alpha, \varepsilon} + \big| \Ds \overline{w}_{\alpha, \varepsilon} \big|^{p - 1} \Ds \overline{w}_{\alpha, \varepsilon}
\ge - C_\star m \varepsilon \, \delta^{\alpha-2} + \Big( \frac{m}{2 C_\star} \Big)^p \, \delta^{- 2 \s p} \geq \\
\ge \Big( \frac{m}{2 C_\star} \Big)^p \, \delta^{- 2 \s p} \bigg( 1 - \frac{2^p C_\star^{p + 1} \diam(\Omega)^{2 (\s p - 1) + \alpha}}{m^{p - 1}} \, \varepsilon \bigg) \ge 0
\end{multline*}
if~$\varepsilon$ is small enough, depending on~$N$,~$\s$,~$p$,~$\Omega$, and~$m$ only. Note that the second inequality holds since, by assumption,~$\sigma p \ge 1$.

We have therefore proved the validity of~\eqref{walphaepssuper}, and thus that~$\overline{w}_{\alpha, \varepsilon}$ is a supersolution to problem~\eqref{probf=0} for any~$\alpha \in (0, \s)$ and any~$\varepsilon \in (0, \varepsilon_0]$, with~$\varepsilon_0 \in (0, 1)$ independent of~$\alpha$. Suppose now that there exists a solution~$u \in L^\infty(\R^N) \cap C^2(\Omega) \cap C^0(\overline{\Omega})$ of~\eqref{probf=0}. By the weak comparison principle of Lemma~\ref{WCPprop}, we then deduce that
\[
u(x) \le \overline{w}_{\alpha, \varepsilon}(x) \quad \mbox{for any } x \in\Omega, \, \alpha \in (0, \s), \mbox{ and } \varepsilon \in (0, \varepsilon_0].
\]
By taking the limit as~$\alpha\downarrow 0$, this in turn implies that
\[
u(x) \le m (1 - \varepsilon_0) \quad \mbox{for any } x \in \Omega.
\]
In particular, since~$m > 0$, we infer that
\[
\sup_{\Omega} u \le m (1 - \varepsilon_0) 
= (1 - \varepsilon_0) \max_{\partial \Omega} g 
< \max_{\partial \Omega} g,
\]
which contradicts the fact that~$u$ attains continuously the boundary datum~$g$.
\end{proof}

Next, we establish the non-existence of solutions also in the case when the right-hand side is too singular at the boundary.

\begin{proof}[Proof of Theorem~\ref{thm:non-ex-sourc}]
For~$\alpha \in (0, \s)$, let~$v_\alpha$ be as in Lemma~\ref{lem:valpha}, and define
\[
\underline{w}_{\, \alpha, \varepsilon} := \varepsilon v_\alpha,
\]
for any~$\varepsilon \in (0, 1)$. With the help of~\eqref{laplvalpha},~\eqref{fraclapvalpha}, and the fact that~$\s p < 1$, we compute
\begin{multline*}
{- \lapl} \underline{w}_{\, \alpha, \varepsilon} +{|\Ds \underline{w}_{\, \alpha, \varepsilon}|}^{p-1}\Ds \underline{w}_{\, \alpha, \varepsilon} \le
C_2 \eps\delta^{\alpha-2} + C_2^p \eps^{p} \diam(\Omega)^{\alpha p} \delta^{-2\s p} \\
\le \eps C_2^p \Big( \diam(\Omega)^\alpha + \diam(\Omega)^{\alpha p + 2 (1 - \s p)} \Big) \delta^{-2} \le \kappa \delta^{-2}\leq f
\end{multline*}
provided~$\eps$ is chosen small enough, depending on~$N$,~$\s$,~$p$,~$\Omega$, and~$\kappa$ only. Note that~$\eps$ can be chosen uniformly with respect to~$\alpha\in(0,\s)$.
Accordingly,~$\underline{w}_{\, \alpha, \varepsilon}$ is a subsolution of problem~\eqref{probg,h=0} for any~$\alpha\in(0,\s)$. By the comparison principle of Proposition~\ref{WCPprop}, we then have that any solution~$u \in L^\infty(\R^N) \cap C^2(\Omega) \cap C^0(\overline{\Omega})$ of~\eqref{probg,h=0} must satisfy~$u\ge \underline{w}_{\, \alpha, \varepsilon} = \eps v_\alpha$ in~$\Omega$. Hence,
\begin{align*}
u(x) \ge \eps \lim_{\alpha\downarrow 0} \tau(x)^\alpha  = \eps \quad \text{for any } x \in \Omega,
\end{align*}
in contradiction with the fact that~$u \in C^0(\overline{\Omega})$ and the homogeneous boundary condition in~\eqref{probg,h=0}.
\end{proof}

\section{Boundary blow-up solutions. Proof of Theorem~\ref{theo:main}} \label{sec:large}

Here, we construct solutions of problem~\eqref{large-prob} which blow up at the boundary of~$\Omega$, thus establishing Theorem~\ref{theo:main}. We will do this by first solving approximating Dirichlet problems with larger and larger data on~$\partial \Omega$ and then passing to limit. This last step will be possible thanks to the barriers provided by the following preliminary result.

\begin{lemma}\label{lem:super}
Let~$\Omega \subseteq \R^N$ be a bounded domain with boundary of class~$C^2$. For
\begin{align*}
p\in\bigg(\frac{3-\s}{1+\s},\frac1\s\bigg),
\end{align*} 
let
\begin{equation} \label{gammadef}
\gamma = \gamma(\s,p) := - \frac{2 (1 - \s p)}{p-1} \in (- 1 + \s, 0)
\end{equation}
and~$V_{\gamma}$ be defined as in~\eqref{Vbetadef}. Then, there exist two constants~$A,B \ge 1$, depending only on~$N$,~$\Omega$,~$\s$, and~$p$, such that the~$L^1(\R^N) \cap C^2(\Omega)$ function
\begin{equation*}
\overline{u}(x) := A V_\gamma(x) + B \chi_\Omega(x), \quad x \in \R^N,
\end{equation*}
satisfies
\begin{align*}
-\lapl\overline{u}+\big|\Ds\overline{u}\big|^{p-1}\Ds\overline{u} \ \ge\ 0
\quad \text{in } \Omega.
\end{align*}
\end{lemma}
\begin{proof}
In light of estimates~\eqref{laplVbeta}-\eqref{fraclapVbeta} of Lemma~\ref{Vbetalem} and~\eqref{fraclapv0} of Lemma~\ref{lem:valpha} (recall that~$v_0 = \chi_{\Omega}$ a.e.~in~$\R^N$), we have that
\[
- \lapl \overline{u} \ge - C_\sharp A \delta^{\gamma - 2}
\]
and
\[
\Ds \overline{u} \ge C_\sharp^{-1} A \delta^{\gamma - 2 \s} + C_3^{-1} B \delta^{-2 \s} \ge C_\sharp^{-1} A \delta^{\gamma - 2 \s}
\qquad\text{in }\Gamma = \{ x \in \Omega : \delta(x) < \delta_1 \}.
\]
Since, by~\eqref{gammadef}, we have~$\gamma - 2 = (\gamma  -2 \s) p$, the above two inequalities give that
\begin{align*}
-\lapl\overline{u}+\big|\Ds\overline{u}\big|^{p-1}\Ds\overline{u} & \ge C_\sharp^{-p} A^p \delta^{(\gamma - 2 \s) p} \Big( 1 - C_\sharp^{p + 1} A^{1 - p} \Big) \ge 0 \quad \mbox{in } \Gamma,
\end{align*}
provided~$A$ is large enough. The fact that, for~$B$ large, the same inequality also holds in~$\Omega \setminus \Gamma$ is a simple consequence of the left-hand bound in~\eqref{fraclapv0} and of the~$C^2(\Omega) \cap L^1(\R^N)$ regularity of~$V_\gamma$---which yields in particular that~$-\lapl V_\gamma$ and~$\Ds V_\gamma$ are both bounded in~$\Omega \setminus \Gamma$.
\end{proof}

\begin{proof}[Proof of Theorem~\ref{theo:main}]
For any~$j\in\N$, consider the solution~$u_j \in L^\infty(\R^N) \cap C^2(\Omega) \cap C^0(\overline{\Omega})$ of problem~\eqref{prob} associated to~$g\equiv j$ on~$\partial\Omega$,~$f\equiv 0$ in~$\Omega$,
and~$h \equiv 0$ in~$\R^N\setminus\overline{\Omega}$---its existence and uniqueness is guaranteed by Theorem~\ref{thm:ex}. By the comparison principle of Proposition~\ref{WCPprop},~${(u_j)}_{j\in\N}$ is a non-decreasing sequence bounded above by the function~$\overline{u}$ of Lemma~\ref{lem:super}. Let now~$j\uparrow\infty$ to get that~${(u_j)}_{j\in\N}$ converges monotonically to some~$u\leq\overline{u}$. We will show that this pointwise limit is the sought solution.

Our argument is similar to the one displayed in Subsections~\ref{subsub:uni}-\ref{subsub:lim}. Let~$q > 1$ and~$\theta \ge N p q$ to be chosen later. By Proposition~\ref{lotest} and Lemma~\ref{DsleH2slem} (applied with~$p q$ in place of~$p$,~$\theta + 2 q$ in place of~$\theta$,~$r = 1$, and~$\varepsilon = 2(1 - \s p) q > 0$), we have
\begin{align*}
\|u_j\|_{L^{2,q}_\theta(\Omega)} \le C\Big(\|\Ds u_j\|^p_{L^{p q}_{\theta + 2 q}(\Omega)} + \|u_j\|_{L^q_\theta(\Omega)} \Big) 
\le C\Big(\|u_j\|^p_{L^{2\s,p q}_{\theta}(\Omega)} + \|u_j\|^p_{L^1(\Omega)} + \|u_j\|_{L^q_\theta(\Omega)} \Big),
\end{align*}
for some constant~$C > 0$ depending only on~$N$,~$p$,~$q$,~$\theta$,~$\s$, and~$\Omega$. In view of Proposition~\ref{lotprops}$.v$ and~\cite{lunardi}*{Corollary~2.1.8}, we estimate
\[
\|u_j\|_{L^{2\s,p q}_{\theta}(\Omega)} \le C \|u_j\|_{L^{2, q}_{\theta}(\Omega)}^{\s} \|u_j\|_{L^{\frac{p q(1 - \s)}{1 - \s p}}_{\theta}(\Omega)}^{1 - \s}.
\]
Thus, using the weighted Young's inequality along with the facts that~$\sigma p < 1$,~$p \ge 1$, and~$0 \le u_j \le \overline{u}$,
\begin{multline*}
\|u_j\|_{L^{2,q}_\theta(\Omega)} 
\le C \bigg( \| u_j \|_{L^{\frac{p q(1 - \s)}{1 - \s p}}_{\theta}(\Omega)}^{\frac{(1 - \s) p}{1 - \s p}} + \| u_j \|^p_{L^1(\Omega)} + \| u_j \|_{L^q_\theta(\Omega)} \bigg) \le \\
\le C \bigg( \| \overline{u} \|_{L^{\frac{p q(1 - \s)}{1 - \s p}}_{\theta}(\Omega)}^{\frac{(1 - \s) p}{1 - \s p}} + \| \overline{u} \|^p_{L^1(\Omega)} + 1\bigg).
\end{multline*}
Notice that the first term involving~$\overline{u}$ on the right-hand side is finite, provided~$\theta$ is taken sufficiently large in dependence of~$N$,~$p$,~$q$, and~$\s$ only, whereas the second term is always finite, as~$\overline{u} \in L^1(\Omega)$.

Since the last estimate holds for every~$q > 1$, by compactness we deduce that~${(u_j)}_{j\in\N}$ actually converges to~$u$ in~$C^{1, \alpha}_\loc(\Omega)$, for every~$\alpha \in (0, 1)$. Using this, it is easy to see that~$u$ satisfies
\[
\int_{\Omega} \Big( \grad u \cdot \grad \varphi + \big( |\Ds u|^{p - 1} \Ds u\big) \varphi \Big)= 0 \quad \mbox{for all } \varphi \in C^\infty_c(\Omega).
\]
By standard elliptic regularity, we then get that~$u \in C^2(\Omega)$ and solves the equation in the pointwise sense.

Estimate~\eqref{ugrowth} is an immediate consequence of the pointwise inequalities~$0 \le u \le \overline{u}$. The fact that~$u > 0$ in~$\Omega$ follows from a simple strong maximum principle. Finally, for all~$x_0 \in \partial \Omega$ we have
\[
\liminf_{\Omega \ni x \rightarrow x_0} u(x) \ge \sup_{j \in \N}\ \lim_{\Omega \ni x \rightarrow x_0} u_j(x)  = \sup_{j \in \N} j = +\infty,
\]
and the proof is complete.
\end{proof}

\section{Comments and open questions}

We conclude the paper with a couple of remarks on possible extensions of our results and on points left open by our analysis.

\begin{enumerate}[label=$\roman*)$,leftmargin=3em]
\item Though stated for the specific operator~$u \longmapsto -\lapl u +\big|\Ds u \big|^{p-1}\Ds u$, the main results of this paper are actually valid for a larger class of operators having~$p$-growth in~$\Ds u$ and satisfying the comparison principle of Proposition~\ref{WCPprop}. While the~$p$-growth structure clearly cannot be fully abandoned---in light of the existence/non-existence dichotomy provided by Theorems~\ref{thm:ex} and~\ref{thm:non-ex}---, it would be nice to understand whether our results could be extended to operators which do not satisfy the hypotheses of Proposition~\ref{WCPprop}, such as~$u \longmapsto -\lapl u + \big|\Ds u \big|^p$ or~$u \longmapsto -\lapl u -\big|\Ds u \big|^{p-1}\Ds u$.
For some more details regarding this issue in the case~$p = 1$, see also~\cite{bdvv}*{Appendix A}.

\item Theorem~\ref{theo:main} gives the existence of a solution to~$-\lapl u +\big|\Ds u \big|^{p-1}\Ds u = 0$ in~$\Omega$ which vanishes a.e.~outside of~$\Omega$ and blows up at its boundary, from the inside. As a byproduct of the method of construction, we obtain the upper bound~\eqref{ugrowth} on its blow-up rate. Unfortunately, we are not able to determine neither a corresponding lower bound nor the uniqueness of the solution. We believe it would be interesting to investigate both these issues.
\end{enumerate}

\end{document}